\documentclass{article}[a4paper]
\usepackage{amsmath,amssymb,amsthm}
\usepackage{enumitem}
\usepackage{mathtools} %coloneq command
\usepackage{pgfplots}

\pgfplotsset{compat=1.18}

\newtheorem{theorem}{Theorem}[section]	
\newtheorem{lemma}[theorem]{Lemma}
\newtheorem{corollary}[theorem]{Corollary}
\newtheorem{proposition}[theorem]{Proposition}

\theoremstyle{definition}

\newtheorem{remark}[theorem]{Remark}
\newtheorem{example}[theorem]{Example}

\newcommand{\imp}{\mathbin{\to}}

\begin{document}
\title{Conuclei on varieties of hoops}
	
\author{Sebasti\'an Buss, Diego Casta\~no, Jos\'e Patricio D\'iaz Varela}
		
\maketitle
	
\begin{abstract}
A conucleus $\delta$ on a partially ordered monoid $\mathbf{A}$ is an interior operator that satisfies $\delta(a) \cdot \delta(b) \leq \delta(a \cdot b)$ and $\delta(a) \cdot \delta(1) = \delta(a)$ for all $a,b \in A$. A conucleus is multiplicative if the equality $\delta(a \cdot b) = \delta(a) \cdot \delta(b)$ holds for all $a,b \in A$. In this article we focus on the study of conuclei on hoops, structures which generalize well-known classes of algebras, such as the class of MV-algebras and BL-algebras. Among several results, we provide a Glivenko-type theorem for conuclei. Special emphasis is given to term definable conuclei. The main result of this article is an explicit description of all terms that define a multiplicative conucleus on every structure of an arbitrary proper variety of Wajsberg hoops. We also show that the problem of finding terms that define (multiplicative) conuclei on a variety of basic hoops or BL-algebras is equivalent to finding such terms on some variety or some pair of varieties of Wajsberg hoops. We provide nontrivial interesting examples.
\end{abstract}
	
\section{Introduction}

A \emph{conucleus} on a partially ordered monoid $\mathbf{A} = \langle A,\leq, \cdot,1 \rangle$ is an interior operator $\delta \colon A \to A$ such that $\delta(a) \cdot \delta(1) = \delta(a)$ and $\delta(a) \cdot \delta(b) \leq \delta(a \cdot b)$ for all $a,b \in A$. The image $\delta(A)$ can be endowed with a partially ordered monoid structure: the \emph{conuclear image} of $\mathbf{A}$ with respect to $\delta$ is the partially ordered monoid $\mathbf{A}_\delta \coloneq \langle \delta(A), \leq, \cdot, \delta(1) \rangle$. The notion of a conucleus can be extended to several classes of structures with a partially ordered monoid reduct. 

In the context of residuated lattices, the concept of a conucleus was introduced by Galatos \cite{GALATOS2003}. In the aforementioned work, the notion of a \emph{kernel}, that is, a conucleus such that $\delta(1) = 1$ and $\delta(\delta(x) \land y) = \delta(x) \land y$, is used to obtain representation results for GMV-algebras. The kernel construction also generalizes the notion of the negative cone of abelian lattice-ordered groups, which was utilized extensively in the setting of MV-algebras.

Another important application of conuclei is the construction known as \emph{poset product}, introduced with the name of \emph{poset sum} in \cite{JIPSEN2009-1}. A poset product of a family of residuated lattices is a specific conuclear image of a direct product of these algebras \cite[Lemma~9.4]{JIPSEN2010}. Poset products have been used to obtain representation results for several classes of GBL-algebras. For a detailed study, we refer the reader to \cite{JIPSEN2009-1,JIPSEN2009-2,JIPSEN2010}. The conucleus that defines poset products has the particularity of satisfying the equality $\delta(x \cdot y) = \delta(x) \cdot \delta(y)$. This property will play an important role in this work. A conucleus that satisfies this equality will be called a \emph{multiplicative conucleus}.

In the framework of residuated lattices, conuclei were also used to represent \emph{Nelson conucleus algebras} \cite{BUSANICHE2022}, a class of algebras that generalizes Kalman residuated lattices, Nelson residuated lattices and Nelson paraconsistent residuated lattices. Nelson conucleus algebras are (cyclic) involutive resiudated lattices equipped with a special kind of multiplicative conucleus, called \emph{Nelson conucleus}. In the case of the aforementioned examples, the corresponding Nelson conucleus can be defined with a term. The representation of a conucleus with a term will be another important property that will be studied in this work.  

In this article we focus on conuclei defined on hoops, with special emphasis in conuclei given by terms on varieties of Wajsberg hoops, basic hoops and BL-algebras. We start with some preliminaries in Section~\ref{sec:prelim}. We continue with some basic results in Section~\ref{sec:conuclei_on_hoops}. The main result of this section is the analogous of the Glivenko property of nuclei (see, for example, \cite{ZHAO2013}), but in the context of conuclei (Theorem~\ref{thm:conuclei_glivenko}). In Section~\ref{sec:conuclei_defined_by_terms}, we turn our focus to the study of terms that define conuclei on varieties of hoops and bounded hoops. We start with examples of varieties where the only term that defines a conucleus on said varieties is (up to equivalence) the identity term $t(x) \coloneq x$ (Proposition~\ref{prop:WH_conuclei}, Proposition~\ref{prop:cancellative_conuclei}). We end this section exhibiting all terms that define (up to equivalence) conuclei on the variety generated by the six element Wajsberg chain. This example shows that describing all conuclei on an arbitrary variety of Wajsberg hoops can result in a nontrivial task. In view of this, we shift our focus to the study of terms that define multiplicative conuclei on varieties of Wajsberg hoops, which will be the main topic in Section~\ref{sec:multiplicative_Wajsberg_conuclei}. The main result of this section is a description of all terms that define multiplicative conuclei on an arbitrary proper variety of Wasjberg hoops (Theorem~\ref{thm:proper_WH_multiplicative_conuclei}). We also show that there exists an effective procedure to calculate (up to equivalence) all multiplicative conuclei on said variety (Theorem~\ref{thm:multiplicative_separation_conucleus}). Furthermore, the number of terms that define a multiplicative conucleus on a proper variety of Wajsberg hoops is (up to equivalence) finite. We end this work with Section~\ref{sec:conuclei_on_BH_and_BL}, where we show that the problem of finding terms that define (multiplicative) conuclei on varieties of basic hoops or varieties of BL-algebras, reduces to the problem of finding such terms on some variety or some pair of varieties of Wajsberg hoops (Theorem~\ref{thm:basic_varieties_conuclei} and Theorem~\ref{thm:BL_conuclei}). 

\section{Preliminaries}\label{sec:prelim}

In this section, we provide basic definitions and examples that are going to be used throughout this paper, together with some general results. For notation and concepts of universal algebra used in this article, we refer the reader to \cite{BURRIS1981}.

A \emph{hoop} \cite{FERREIRIM1992} is an algebra $\mathbf{A} \coloneq \langle A, \cdot, \imp, \top \rangle$ such that $\langle A, \cdot, \top \rangle$ is a commutative monoid and the following identities hold
\begin{equation*}
	\begin{aligned}
		x \imp x & \approx \top, \\
		x \cdot(x \imp y) & \approx y \cdot(y \imp x), \\
		x \imp (y \imp z) & \approx (x \cdot y) \imp z.
	\end{aligned}
\end{equation*}
The class of all hoops is a variety that will be denoted by $\mathcal{H}$. As usual, we define $a^0 \coloneq \top$ and $a^{n+1} \coloneq a^n \cdot a$ for all $a \in A$, $n \in \mathbb{N}_0$. 

\begin{proposition}\label{prop:hoop_order}
	Let $\mathbf{A}$ be a hoop. The binary relation
	\begin{equation*}
		x \leq y \Leftrightarrow x \imp y \approx \top
	\end{equation*}
	defines a partial order with a semilattice structure over $\mathbf{A}$, where the infimum is given by $x \land y = x \cdot (x \imp y)$ and $\top$ is the greatest element. Furthermore, 
	\begin{equation*}
		x \cdot y \leq z \Leftrightarrow x \leq y \imp z.
	\end{equation*}
\end{proposition}

A \emph{bounded hoop} is an algebra $\mathbf{A} \coloneq \langle A, \cdot, \imp, \bot, \top \rangle$ such that $\langle A, \cdot, \imp, \top \rangle$ is a hoop and $\bot \leq a$ for each $a \in A$. As usual, we define $\neg a \coloneq a \imp \bot$ for all $a \in A$. The variety of bounded hoops will be denoted by $\mathcal{H}_\bot$. If $\mathbf{B} \coloneq \langle B, \cdot, \imp, \top \rangle$ is a hoop with least element $\bot$, we define the bounded hoop $\mathbf{B}^+ \coloneq \langle B, \cdot, \imp, \bot, \top \rangle$. 

In the following result, we list some basic algebraic properties of hoops (see, for example, \cite{FERREIRIM1992}).

\begin{proposition}\label{prop:hoops_properties}
	Let $\mathbf{A}$ be a hoop and $a,b,c \in A$. Then:
	\begin{enumerate}[label=\rm{(\arabic*)}]
		\item $\top \imp a = a$.
		\item If $a \leq b$, then $a \cdot c \leq b \cdot c$, $b \imp c \leq a \imp c$ and $c \imp a \leq c \imp b$.
		\item $a \cdot b \leq a \land b$.
		\item $b \leq a \imp (a \cdot b) \leq a \imp b$.
		\item $a \imp (b \imp c) = (a \cdot b) \imp c = b \imp (a \imp c)$.
		\item If $a \lor b$ exists, then $a \lor b \leq (a \imp b) \imp b$.
		\item $((a \imp b) \imp b) \imp b = a \imp b$.
	\end{enumerate}
\end{proposition}

When the order of a hoop $\mathbf{A}$ is total, we say that $\mathbf{A}$ is a \emph{totally ordered hoop}. A \emph{Wajsberg hoop} is a hoop that satisfies the identity
\begin{equation*}\label{tanaka}
	(x \imp y) \imp y \approx (y \imp x) \imp x.
\end{equation*}
Let $\mathcal{WH}$ denote the variety of Wajsberg hoops. For every Wajsberg hoop, the natural hoop order defines a lattice structure, where the supremum is given by the term $x \lor y \coloneq (x \imp y) \imp y$.

A hoop is \emph{cancellative} if it satisfies the identity
\begin{equation*}
	x \imp (x \cdot y) \approx y.
\end{equation*}
The following result is a well-known property of totally ordered Wajsberg hoops (see, for example, \cite[Lemma~4.5, Proposition~4.6]{FERREIRIM1992}). We say that a hoop $\mathbf{A}$ is \emph{semi-cancellative} if $c < a \cdot b$ implies $a \imp (a \cdot b) = b$, for all $a,b,c \in A$.

\begin{proposition}\label{prop:to_wajsberg_is_semi_cancellative}
	Every totally ordered Wajsberg hoop is semi-cancellative. As a consequence, a nontrivial totally ordered Wajsberg hoop is either bounded or cancellative. 
\end{proposition}

Let us continue with some well-known examples that are going to be used throughout this work.

\begin{example}\label{ex:hoops}
	Consider $\mathbf{S} \coloneq \langle [0,1], \cdot, \to, 1 \rangle$, where $x \cdot y \coloneq \max \lbrace x + y - 1, 0 \rbrace$ and $x \imp y \coloneq \min \lbrace y - x + 1, 1 \rbrace$. With these definitions, $\mathbf{S}$ is a totally ordered Wajsberg hoop. 
	
	The set $\lbrace 1 \rbrace$ is a subuniverse of $\mathbf{S}$. The algebra associated to this subuniverse, which will be denoted by $\mathbf{S}_0$, is a trivial hoop. For each $n \in \mathbb{N}$, consider the subset of $S$ defined by $S_n \coloneq \lbrace \frac{i}{n}: i \in \lbrace 0,\dots,n \rbrace \rbrace$. Then, $S_n$ is a subuniverse of $\mathbf{S}$ and the structure $\mathbf{S}_n \coloneq \langle S_n, \cdot^\mathbf{S}, \imp^\mathbf{S}, 1 \rangle$ is a totally ordered Wajsberg hoop.
	
	Next, consider the set $\mathbb{N}_0$ with the operation $i \imp j \coloneq \max\lbrace j-i,0 \rbrace$. If $+$ is the usual sum of integers, the structure $\mathbf{S}^\omega \coloneq \langle \mathbb{N}_0, +, \imp, 0 \rangle$ is a cancellative totally ordered Wajsberg hoop.  
	
	Finally, consider $n \in \mathbb{N}$ and the closed interval $S_n^\omega \coloneq [(0,0), (n,0)]_{\mathsf{lex}}$ of $\mathbb{Z}^2$ (where $\mathsf{lex}$ denotes the lexicographic order of $\mathbb{Z}^2$), with $x \cdot y \coloneq \neg(\neg x \oplus \neg y)$ and $x \imp y \coloneq \neg x \oplus y$, where $x \oplus y \coloneq (x + y) \land (n,0)$ and $\neg x \coloneq (n,0) - x$, with $+$ and $-$ the addition and subtraction operations on $\mathbb{Z}^2$, respectively. The algebra $\mathbf{S}_n^\omega \coloneq \langle S_n^\omega, \cdot, \imp, \top \rangle$ is a totally ordered Wajsberg hoop.		
\end{example}

We recall a well-known result about varieties of Wajsberg hoops. Let $I,J$ be finite subsets of $\mathbb{N}$ and let $K \subseteq \lbrace 0 \rbrace$. A triple $(I,J,K)$ is \emph{reduced} if
\begin{enumerate}[label=-]
	\item $I \cup J \cup K \neq \emptyset$;
	\item if $K = \lbrace 0 \rbrace$, then $J = \emptyset$;
	\item for all $i \in I$ and $i' \in (I \backslash \lbrace i \rbrace) \cup J$, $i$ does not divide $i'$;
	\item for all $j \in J$ and $j' \in J \backslash \lbrace j \rbrace$, $j$ does not divide $j'$.
\end{enumerate}
For a class of (bounded) hoops $\mathcal{K}$, let $\mathsf{V}(\mathcal{K})$ denote the variety generated by $\mathcal{K}$. The following facts can be found in \cite{AGLIANO2002}.

\begin{theorem}\label{thm:proper_WH_varieties} $ $
	\begin{enumerate}[label=\rm{(\arabic*)}]
		\item $\mathcal{WH} = \mathsf{V}(\mathbf{S})$.
		\item The proper subvarieties of $\mathcal{WH}$ are in a 1-1 correspondence with the reduced triples, via the mappings
		\begin{equation*}
			\begin{aligned}
				(I,J,\emptyset) & \mapsto \mathsf{V}(\lbrace \mathbf{S}_i: i \in I \rbrace \cup \lbrace \mathbf{S}_j^\omega: j \in J \rbrace); \\
				(I, \emptyset,\lbrace 0 \rbrace) & \mapsto \mathsf{V}(\lbrace \mathbf{S}_i: i \in I \rbrace \cup \lbrace \mathbf{S}^\omega \rbrace).
			\end{aligned}	
		\end{equation*}
	\end{enumerate}
\end{theorem}

A \emph{basic hoop} \cite{AGLIANO2007} is a hoop which satisfies the identity
\begin{equation*}
	((x \imp y) \imp z) \imp (((y \imp x) \imp z) \imp z) \approx \top.
\end{equation*}
Let $\mathcal{BH}$ denote the variety of basic hoops. The finitely subdirectly irreducible members of $\mathcal{BH}$ are the nontrivial totally ordered hoops  \cite[Proposition~1.8]{AGLIANO2007}. Particularly, $\mathcal{WH}$ is a subvariety of $\mathcal{BH}$. Given a variety $\mathcal{V}$ of basic hoops, let $\mathcal{V}_{\textnormal{fsi}}$ denote the class of all finitely subdirectly irreducible members of $\mathcal{V}$.

A \emph{BL-algebra} is a bounded basic hoop. On the other hand, an \emph{MV-algebra} is a bounded Wajsberg hoop. The varieties of BL-algebras and MV-algebras will be denoted by $\mathcal{BL}$ and $\mathcal{MV}$, respectively. Clearly, $\mathcal{MV}$ is a subvariety of $\mathcal{BL}$. As is the case with basic hoops, the finitely subdirectly irreducible members of $\mathcal{BL}$ are the nontrivial BL-chains, that is, the nontrivial totally ordered BL-algebras. Given a variety $\mathcal{V}$ of BL-algebras, we also let $\mathcal{V}_{\textnormal{fsi}}$ denote the class of all finitely subdirectly irreducible members of $\mathcal{V}$.

Let $\langle I, \leq \rangle$ be a total order. For each $i \in I$ let $\mathbf{A}_i \coloneq \langle A_i, \cdot_i, \imp_i, \top \rangle$ be a hoop such that $A_i \cap A_j = \lbrace \top \rbrace$ for all $i,j \in I$, $i \neq j$. The \emph{ordinal sum} is the hoop $\bigoplus_{i \in I} \mathbf{A}_i \coloneq \langle \bigcup_{i \in I} A_i, \cdot, \imp, \top \rangle$, where the operations are defined as follows:
\begin{equation*}
	x \cdot y \coloneq
	\begin{cases}
		x \cdot_i y & \text{if $x,y \in A_i$}; \\
		x & \text{if $x \in A_i \backslash \lbrace \top \rbrace$, $y \in A_j$, $i < j$}; \\
		y & \text{if $y \in A_i \backslash \lbrace \top \rbrace$, $x \in A_j$, $i < j$};
	\end{cases}
\end{equation*}
and
\begin{equation*}
	x \imp y \coloneq
	\begin{cases}
		\top & \text{if $x \in A_i \backslash \lbrace \top \rbrace$, $y \in A_j$, $i < j$}; \\
		x \imp_i y & \text{if $x,y \in A_i$}; \\
		y & \text{if $y \in A_i$, $x \in A_j$, $i < j$}.
	\end{cases}
\end{equation*}
Each $\mathbf{A}_i$ will be called \emph{component} of $\bigoplus_{i \in I} \mathbf{A}_i$. Ordinal sums play a key role in the description of totally ordered hoops. The following decomposition is a well-known result in the theory of hoops (see, for example, \cite[Theorem~3.7]{AGLIANO2003}).

\begin{theorem}\label{thm:totally_ordered_hoops_decomposition}
	For each nontrivial totally ordered hoop $\mathbf{A}$ there exists a total order $\langle I, \leq \rangle$ and a set of nontrivial totally ordered Wajsberg hoops $\lbrace \mathbf{A}_i: i \in I \rbrace$ such that $A_i \cap A_j = \lbrace \top \rbrace$ for all $i \neq j$ and $\mathbf{A} = \bigoplus_{i \in I} \mathbf{A}_i$. Conversely, if $\langle I, \leq \rangle$ is a total order and $\lbrace \mathbf{A}_i: i \in I \rbrace$ is a set of nontrivial totally ordered Wajsberg hoops such that $A_i \cap A_j = \lbrace \top \rbrace$ for all $i \neq j$, then $\bigoplus_{i \in I} \mathbf{A}_i$ is a nontrivial totally ordered hoop.
\end{theorem}	

\begin{remark}\label{rmk:BL_chains_decomposition}
	Since every BL-chain is a totally ordered hoop, we have the following version of Theorem~\ref{thm:totally_ordered_hoops_decomposition} for BL-chains: a bounded hoop $\mathbf{A}$ is a nontrivial BL-chain if and only if there exists a total order $\langle I, \leq\rangle$ with least element $0$ and a set of nontrivial totally ordered Wajsberg hoops $\lbrace \mathbf{A}_i: i \in I \rbrace$ such that $A_i \cap A_j = \lbrace \top \rbrace$ for all $i \neq j$, $\mathbf{A}_0$ has least element and $\mathbf{A} = (\bigoplus_{i \in I} \mathbf{A}_i)^+$.
\end{remark}

Given $\lbrace \mathbf{A}_i: i \in I \rbrace$ a set of hoops, it is always possible to find another set of hoops $\lbrace \mathbf{B}_i: i \in I \rbrace$ such that $\mathbf{A}_i$ is isomorphic to $\mathbf{B}_i$ and $B_i \cap B_j = \lbrace \top \rbrace$ for all $i,j \in I$, $i \neq j$. Therefore, for simplicity, we can always safely assume that the condition $A_i \cap A_j = \lbrace \top \rbrace$ for all $i \neq j$ is true. As an example, the structures $\mathbf{S}_n \oplus \mathbf{S}_n$ and $\mathbf{S}_n \oplus \mathbf{S}^\omega$ are going to be considered well-defined totally ordered hoops. Furthermore, if $\mathbf{A}$ is a nontrivial totally ordered hoop (BL-chain) and we write $\mathbf{A} = \bigoplus_{i \in I} \mathbf{A}_i$ ($\mathbf{A} = (\bigoplus_{i \in I} \mathbf{A}_i)^+$), we assume in the sequel that the total order $\langle I,\leq \rangle$ and the familiy $\lbrace \mathbf{A}_i: i \in I \rbrace$ are those given by Theorem~\ref{thm:totally_ordered_hoops_decomposition} (Remark~\ref{rmk:BL_chains_decomposition}).

We continue this section of preliminaries with some facts about one generated terms in the language $\lbrace \cdot,\imp,\top \rbrace$ and $\lbrace \cdot,\imp,\bot,\top \rbrace$. 
\begin{proposition} $ $
	\begin{enumerate}[label=\rm{(\arabic*)}]
		\item If $t(x)$ is a term in the language $\lbrace \cdot,\imp,\top \rbrace$, then $\mathcal{H} \vDash t(\top) \approx \top$.
		\item If $t(x)$ is a term in the language $\lbrace \cdot,\imp,\bot,\top \rbrace$, then there exist $c,d \in \lbrace \bot,\top \rbrace$ such that $\mathcal{H}_\bot \vDash t(\bot) \approx c$ and $\mathcal{H}_\bot \vDash t(\top) \approx d$.
	\end{enumerate}
\end{proposition}

In view of the last result and with the aim of simplifying the notation, given a term $t(x)$ in the language $\lbrace \cdot,\imp,\top \rbrace$ or $\lbrace \cdot,\imp,\bot,\top \rbrace$, we are sometimes going to write $t(c) \approx d$ with $c,d \in \lbrace \bot,\top \rbrace$, to indicate that $\mathbf{A} \vDash t(c) \approx d$ for each hoop or bounded hoop $\mathbf{A}$ (as appropriate).

The next result is a well-known fact about terms in the language $\lbrace \cdot,\imp,\bot,\top \rbrace$, in the context of MV-algebras (see \cite[Lemma 1]{AGUZZOLI2010}).

\begin{proposition}\label{prop:MVterm_description}
	If $t(x)$ is a term in the language $\lbrace \cdot,\imp,\bot,\top \rbrace$ and $t(\top) \approx \top$, then there exists a term $\overline{t}(x)$ in the language $\lbrace \cdot,\imp,\top \rbrace$, which can be effectively calculated, such that $\mathbf{S}^+ \vDash t(x) \approx \overline{t}(x)$ (equiv. $\mathcal{MV} \vDash t(x) \approx \overline{t}(x)$).
\end{proposition}

For this work we are going to need a result similar to the previous one, but in the setting of BL-algebras. As we are going to show next, we can formulate a similar property for BL-chains if it is the case that $t(\top) \approx \top$. We start with some preliminaries.

\begin{lemma}\label{lemma:BLterm_description}
	If $t(x)$ is a term in the language $\lbrace \cdot,\imp,\bot,\top \rbrace$, then there exists a term $\overline{t}(x)$ such that:
	\begin{enumerate}[label=\rm{(\arabic*)}]
		\item $\overline{t}(x) = \bot$ or $\overline{t}(x)$ is a term in the language $\lbrace \cdot,\imp, \top \rbrace$.
		\item $\ensuremath{\left. t^\mathbf{A} \right|_{A_i}} = \ensuremath{\left. \overline{t}^\mathbf{A} \right|_{A_i}}$ for all $\mathbf{A} \coloneq (\bigoplus_{i \in I} \mathbf{A}_i)^+ \in \mathcal{BL}_{\textnormal{fsi}}$ and $i \in I \backslash \lbrace 0 \rbrace$.
	\end{enumerate}
\end{lemma}
\begin{proof}
	The proof will be done using induction over the complexity of $t(x)$. If $t(x) \in \lbrace x,\bot,\top \rbrace$ the result is immediate. Now assume that $t(x) = t_1(x) \star t_2(x)$ for some $\star \in \lbrace \cdot,\imp \rbrace$. Using the inductive hypothesis, we obtain $\overline{t_1}(x)$ and $\overline{t_2}(x)$ terms such that, for each $k \in \lbrace 1,2 \rbrace$,
	\begin{enumerate}[label=-]
		\item $\overline{t_k}(x) = \bot$ or $\overline{t_k}(x)$ is a term in the language $\lbrace{\cdot,\imp,\top\rbrace}$ and
		\item $\ensuremath{\left. t_k^\mathbf{A} \right|_{A_i}} = \ensuremath{\left. \overline{t_k}^\mathbf{A} \right|_{A_i}}$ for all $\mathbf{A} \coloneq (\bigoplus_{i \in I} \mathbf{A}_i)^+ \in \mathcal{BL}_{\textnormal{fsi}}$ and $i \in I \backslash \lbrace 0 \rbrace$.
	\end{enumerate}
	
	If $\star = \cdot$, we define
	\begin{equation*}
		\overline{t}(x) \coloneq
		\begin{cases}
			\bot & \text{if $\overline{t_1}(x) = \bot$ or $\overline{t_2}(x) = \bot$;} \\
			\overline{t_1}(x) \cdot \overline{t_2}(x) & \text{if $\overline{t_1}(x)$ and $\overline{t_2}(x)$ are $\lbrace \cdot,\imp,\top \rbrace$-terms.}
		\end{cases}
	\end{equation*} 	
	
	On the other hand, if $\star = \imp$, we define	
	\begin{equation*}
		\overline{t}(x) \coloneq
		\begin{cases}
			\top & \text{if $\overline{t_1}(x) = \bot$;} \\
			\bot & \text{if $\overline{t_1}(x)$ is a $\lbrace \cdot,\imp,\top \rbrace$-term and $\overline{t_2}(x) = \bot$;} \\
			\overline{t_1}(x) \imp \overline{t_2}(x) & \text{if $\overline{t_1}(x)$ and $\overline{t_2}(x)$ are $\lbrace \cdot,\imp,\top \rbrace$-terms.}
		\end{cases}
	\end{equation*}
	
	In both cases, $\overline{t}(x)$ satisfies conditions (1) and (2).  
\end{proof}

One of the problems in the previous lemma is that it does not state anything about the interpretation of the term in the first component of the decomposition of a chain. This can be easily fixed if $t(\top) \approx \top$, combining Proposition~\ref{prop:MVterm_description} and Lemma~\ref{lemma:BLterm_description}. Recall that the first component of the decomposition of a BL-chain is a subalgebra that is an MV-algebra.

\begin{corollary}\label{cor:BL_terms}
	Let $t(x)$ be a term in the language $\lbrace \cdot,\to,\bot,\top \rbrace$. If $t(\top) \approx \top$, then there exist terms $s(x)$ and $u(x)$ in the language $\lbrace \cdot,\imp,\top \rbrace$ such that, for all $\mathbf{A} \coloneq (\bigoplus_{i \in I} \mathbf{A}_i)^+ \in \mathcal{BL}_{\mathrm{fsi}}$,
	\begin{enumerate}[label=\rm{(\arabic*)}]
		\item $\ensuremath{\left. t^\mathbf{A} \right|_{A_0}} = \ensuremath{\left. s^\mathbf{A} \right|_{A_0}} = s^\mathbf{A_0}$;
		\item $\ensuremath{\left. t^\mathbf{A} \right|_{A_i}} = \ensuremath{\left. u^\mathbf{A} \right|_{A_i}} = u^\mathbf{A_i}$ for all $i \in I \backslash \lbrace 0 \rbrace$;
	\end{enumerate}
	Particularly, $t^\mathbf{A}(A_i) \subseteq A_i$ for all $i \in I$.
\end{corollary}

The conditions obtained in the previous corollary can be rewritten. The next equivalence will be an important tool for some of the results in this paper.

\begin{proposition}\label{prop:BLterms_equivalence}
	Let $t(x)$ be a term in the language $\lbrace \cdot,\imp,\bot,\top \rbrace$ such that $t(\top) \approx \top$. Let $s(x)$ and $u(x)$ be terms in the language $\lbrace \cdot,\imp,\top \rbrace$. For all $\mathbf{A} \coloneq (\bigoplus_{i \in I} \mathbf{A}_i)^+ \in \mathcal{BL}_{\textnormal{fsi}}$, the following are equivalent:
	\begin{enumerate}[label=\rm{(\roman*)}]
		\item $\ensuremath{\left. t^{\mathbf{A}} \right|_{A_0}} = \ensuremath{\left. s^{\mathbf{A}} \right|_{A_0}}$ and $\ensuremath{\left. t^{\mathbf{A}} \right|_{A_i}} = \ensuremath{\left. u^{\mathbf{A}} \right|_{A_i}}$ for all $i \in I \backslash \lbrace 0 \rbrace$;
		\item $\mathbf{A} \vDash t(x) \approx s(\neg\neg x) \cdot u(\neg\neg x \imp x)$.
	\end{enumerate}
\end{proposition}
\begin{proof}
	To prove this result observe that, if $a \in A_0$,
	\begin{equation*}
		s^{\mathbf{A}}(\neg\neg a) \cdot u^{\mathbf{A}}(\neg\neg a \imp a) = s^{\mathbf{A}}(a) \cdot u^{\mathbf{A}}(\top) = s^{\mathbf{A}}(a).
	\end{equation*}
	Furthermore, if $a \in A_i$ with $i \neq 0$,
	\begin{equation*}
		s^{\mathbf{A}}(\neg\neg a) \cdot u^{\mathbf{A}}(\neg\neg a \imp a) = s^{\mathbf{A}}(\top) \cdot u^{\mathbf{A}}(a) = u^{\mathbf{A}}(a).
	\end{equation*}  
\end{proof}

A \emph{McNaughton function} over the interval [0,1] is a function $f\colon [0,1] \to [0,1]$ such that:
\begin{enumerate}[label=-]
	\item $f$ is continuous with respect to the natural topology of $[0,1]$;
	\item there are linear functions $p_1,\dots,p_k$, integer coefficients $a_1,\dots,a_k$, $b_1,\dots,b_k$ such that $p_i(x) = a_i  x + b_i$ for all $i \in \lbrace 1,\dots,k \rbrace$ and for every $a \in [0,1]$ there is some index $j \in \lbrace 1,\dots,k \rbrace$ such that $f(a) = p_j(a)$. 
\end{enumerate}
The following result is a well-known fact of one generated terms in the context of Wajsberg hoops and their bounded counterpart, MV-algebras (see \cite{AGLIANO2002,MCNAUGHTON1951}).

\begin{theorem}\label{thm:mcnaughton}
	If $t(x)$ is a term in the language $\lbrace \cdot,\imp,\bot,\top \rbrace$, then $t^{\mathbf{S}^+}$ is a McNaughton function. Conversely, for every McNaughton function $f$ there exists some term $s(x)$ in the language $\lbrace \cdot,\imp,\bot,\top \rbrace$, which can be effectively calculated, such that $s^{\mathbf{S}^+} = f$.
	
	If $t(x)$ is a term in the language $\lbrace \cdot,\imp,\top \rbrace$, then $t^\mathbf{S}$ is a McNaughton function and $t^\mathbf{S}(1) = 1$. Moreover, for every McNaughton function $f$ such that $f(1) = 1$ there exists some term $s(x)$ in the language $\lbrace \cdot,\imp,\top \rbrace$, which can be effectively calculated, such that $s^\mathbf{S} = f$.
\end{theorem}

The sequence of \emph{Farey partitions} of the interval $[0,1]$ is the sequence defined inductively by
\begin{equation*}
	\mathsf{Far}_n \coloneq
	\begin{cases}
		\lbrace 0,1 \rbrace & \text{if $n = 0$}; \\
		\mathsf{Far}_{n-1} \cup \lbrace \frac{a + c}{b + d}: \frac{a}{b}, \frac{c}{d} \text{ are consecutive elements of $\mathsf{Far}_{n-1}$} \rbrace & \text{if $n > 0$};
	\end{cases} 
\end{equation*} 
with $0 \coloneq \frac{0}{1}$ and $1 \coloneq \frac{1}{1}$. For each $n \in \mathbb{N}_0$ and $q \in \mathsf{Far}_n$ let $q^-$ and $q^+$ denote the predecessor and successor (if they exist) of $q$ in $\mathsf{Far}_n$, respectively. We list some important properties of this sequence.

\begin{proposition}\label{prop:farey_prop1}
	$ $
	\begin{enumerate}[label=\rm{(\arabic*)}]
		\item For all $n \in \mathbb{N}_0$, all fractions in $\mathsf{Far}_n$ are in irreducible form and the interval $[\frac{a}{b},\frac{c}{d}]$ determined by any pair of consecutive fractions in $\mathsf{Far}_n$ has the unimodularity property, that is, $bc - ad = 1$. 
		\item For every irreducible fraction $q \in \mathbb{Q} \cap [0,1]$ there exists some $m \in \mathbb{N}_0$ such that $q \in \mathsf{Far}_m$. 
	\end{enumerate}
\end{proposition}

For each $n \in \mathbb{N}_0$ and $a \coloneq \frac{c}{d} \in \mathsf{Far}_n$, the \emph{Schauder hat} of the element $a$ in $\mathsf{Far}_n$, which will be denoted by $h_{n,a}$, is the McNaughton function $h_{n,a} \colon [0,1] \to [0,1]$ such that the graph of $h_{n,a}$ consists of the segments joining the points
\begin{enumerate}[label=-]
	\item $(0,1), (0^+,0), (1,0)$ (if $a = 0$);
	\item $(0,0), (a^-,0), (a,\frac{1}{d}), (a^+,0), (1,0)$ (if $0 < a < 1$);
	\item $(0,0), (1^-,0), (1,1)$ (if $a = 1$).
\end{enumerate}
Before we continue, we consider, for bounded hoops, the term
\begin{equation*}
	x \oplus y \coloneq \neg(\neg x \cdot \neg y).
\end{equation*}
The structure $\langle [0,1], \oplus^{\mathbf{S}^+}, 0 \rangle$ is a commutative monoid (see, for example, \cite{CIGNOLI2000}). Taking this into account, for all $n \in \mathbb{N}_0$ we also define the term
\begin{equation*}
	n.x \coloneq
	\begin{cases}
		\bot & \text{if $n = 0$;} \\
		x \oplus ((n-1).x) & \text{if $n > 0$}.
	\end{cases}
\end{equation*}

We say that $a \in [0,1]$ is a \emph{node} of a continuous function $f\colon [0,1] \to [0,1]$ if it is a nondifferentiability point of $f$. The elements $0$ and $1$ will always be considered nodes of $f$. Observe that the nodes of a McNaughton function are, by definition, a finite subset of $\mathbb{Q}$. For the proof of the following theorem we refer the reader to \cite{MUNDICI1994}. 

\begin{theorem}\label{thm:mcnaughton_constructive}
	Let $f\colon [0,1] \to [0,1]$ be a McNaughton function and let $N = \lbrace a_1,\dots,a_u \rbrace \subseteq \mathbb{Q}$ be the set of nodes of $f$. If $n \in \mathbb{N}_0$ is such that $N \subseteq \mathsf{Far}_n$, then 
	\begin{equation*}
		f = \bigoplus_{a \in \mathsf{Far}_n} k_a. h_{n,a}
	\end{equation*}	
	where, for each $a \coloneq \frac{c}{d} \in \mathsf{Far}_n$, the integer $k_a \in \mathbb{N}_0$ is such that $f(a) = \frac{k_a}{d}$. 
\end{theorem}

The next lemma will be a useful tool for the construction of McNaughton functions with specific properties. On every interval whose bounds are consecutive elements of some Farey set, one can define linear functions with integer coefficients that satisfy special properties. The proof of this fact is a straightforward consequence of the unimodularity property.

\begin{lemma}\label{lemma:farey_prop2}
	Let $n \in \mathbb{N}_0$ and let $\frac{a}{b}$, $\frac{c}{d}$ be consecutive elements of $\mathsf{Far}_n$.
	\begin{enumerate}[label=\rm{(\arabic*)}]
		\item The linear function $f(x) \coloneq -adx + ac$ satisfies $f(\frac{a}{b}) = \frac{a}{b}$, $f(\frac{c}{d}) = 0$;
		\item The linear function $g(x) \coloneq bcx - ac$ satisfies $f(\frac{a}{b}) = 0$, $f(\frac{c}{d}) = \frac{c}{d}$.
	\end{enumerate}
\end{lemma}

We end this section of preliminaries with a useful theorem that will allow us to construct McNaughton functions with certain properties. This result will be crucial to prove one of the main results in this work. 

\begin{theorem}\label{thm:local_identities}
	Let $t(x)$ be a term in the language $\lbrace \cdot,\imp,\top \rbrace$. Consider $m \in \mathbb{N}$, $i \in \lbrace 0,\dots,m \rbrace$, $\varepsilon > 0$ and $k_1,k_2 \in \mathbb{Z}$. 
	\begin{enumerate}[label=\rm{(\arabic*)}]
		\item If $[\frac{i}{m},\frac{i}{m} + \varepsilon) \subseteq [0,1]$ and $t^{\mathbf{S}}(a) = k_1 a + k_2$ for all $a \in [\frac{i}{m},\frac{i}{m} + \varepsilon)$, then $t^{\mathbf{S}_m}(\frac{i}{m}) = k_1\frac{i}{m} + k_2$ and $t^{\mathbf{S}_m^\omega}(i,j) = (k_1 i + k_2 m, k_1 j)$ for all $j \geq 0$.
		\item If $(\frac{i}{m} - \varepsilon,\frac{i}{m}] \subseteq [0,1]$ and $t^{\mathbf{S}}(a) = k_1 a + k_2$ for all $a \in (\frac{i}{m} - \varepsilon,\frac{i}{m}]$, then $t^{\mathbf{S}_m}(\frac{i}{m}) = k_1\frac{i}{m} + k_2$ and $t^{\mathbf{S}_m^\omega}(i,j) = (k_1 i + k_2 m, k_1 j)$ for all $j \leq 0$. 
	\end{enumerate}
\end{theorem}
\begin{proof}
	Let $\mathfrak{U}$ be some non-principal ultrafilter over $\mathbb{N}$. To prove this result, consider the ultraproduct $\mathbf{S}^* \coloneq [\prod_{n \in \mathbb{N}} \mathbf{S}] \slash \mathfrak{U}$. Then, the following hold:
	\begin{enumerate}[label=-]
		\item $\mathbf{S}^*$ is a totally ordered Wajsberg hoop.
		\item If 
		\begin{equation*}
			\varphi_m(i,j) \coloneq \min \left\lbrace n \in \mathbb{N}: \frac{i}{m} + \frac{j}{n} \in [0,1] \right\rbrace
		\end{equation*}
		for all $(i,j) \in S_m^\omega$, then the map $\varphi_m: S_m^\omega \to \mathbb{N}$ is well defined. 
		\item The mapping $f(i,j) \coloneq a(i,j) \slash \mathfrak{U}$ is an embedding from  $\mathbf{S}_m^\omega$ into $\mathbf{S}^*$, where
		\begin{equation*}
			a(i,j)_n \coloneq 
			\begin{cases}
				1 & \text{if $n < \varphi_m(i,j)$;} \\
				\frac{i}{m} + \frac{j}{n} & \text{if $n \geq \varphi_m(i,j)$;} 
			\end{cases}
		\end{equation*}
		for all $(i,j) \in S_m^\omega$ and $n \in \mathbb{N}$. 
	\end{enumerate}
	
	We are going to restrict the proof to item (1). The proof of item (2) is analogous. The equality $t^{\mathbf{S}_m}(\frac{i}{m}) = k_1 \frac{i}{m} + k_2$ is immediate, because $\mathbf{S}_m$ is a subalgebra of $\mathbf{S}$. Given some $j \geq 0$, we first need to show that $(k_1 i + k_2 m,k_1 j)$ is effectively an element of $S_m^\omega$. First observe that
	\begin{equation*}
		k_1\frac{i}{m} + k_2 \in [0,1] \Rightarrow m(k_1\frac{i}{m} + k_2) \in [0,m] \Rightarrow k_1 i + k_2 m \in \lbrace 0,\dots,m \rbrace.
	\end{equation*}
	Taking the previous remarks into account, it suffices to show that $k_1 i + k_2 m = 0$ implies $k_1 j \geq 0$ and $k_1 i + k_2 m = m$ implies $k_1 j \leq 0$. We can safely assume that $j > 0$, since the proof of $j = 0$ is immediate. We start with the case where $k_1 i + k_2 m = 0$. Since $j > 0$, there exists some $n \in \mathbb{N}$ such that $\frac{i}{m} + \frac{j}{n} \in (\frac{i}{m},\frac{i}{m} + \varepsilon)$. Then, $k_1(\frac{i}{m} + \frac{j}{n}) + k_2 \in [0,1]$. If $k_1 < 0$,  
	\begin{equation*}
		\textstyle k_1(\frac{i}{m} + \frac{j}{n}) + k_2 < k_1 \frac{i}{m} + k_2 = 0
	\end{equation*}
	and we reached a contradiction. Therefore $k_1 \geq 0$ and $k_1 j \geq 0$. Now we suppose that $k_1 i + k_2 m = m$. Considering the same $n$ as before, if $k_1 > 0$, then
	\begin{equation*}
		\textstyle k_1(\frac{i}{m} + \frac{j}{n}) + k_2 > k_1 \frac{i}{m} + k_2 = 1
	\end{equation*}
	and we reached a contradiction, again. Then $k_1 \leq 0$ and $k_1 j \leq 0$.
	
	We now continue with the proof of (1). Since $j \geq 0$, there exists some $k \in \mathbb{N}$ such that $\frac{i}{m} + \frac{j}{n} \in [\frac{i}{m},\frac{i}{m} + \varepsilon)$ for all $n \geq k$. If we define $l \coloneq \max\lbrace k,\varphi_m(i,j),\varphi_m(k_1 i + k_2 m,k_1 j) \rbrace$, then
	\begin{equation*}
		\begin{aligned}
			\textstyle t^{\mathbf{S}}(a(i,j)_n) & \textstyle = t^{\mathbf{S}}(\frac{i}{m} + \frac{j}{n}) \\
			& \textstyle = k_1(\frac{i}{m} + \frac{j}{n}) + k_2 \\
			& \textstyle = \frac{k_1i + k_2m}{m} + \frac{k_1j}{n} \\
			& \textstyle = a(k_1 i + k_2 m,k_1 j)_n.
		\end{aligned}
	\end{equation*}
	for all $n \geq l$. Therefore,
	\begin{equation*}
		\lbrace n \in \mathbb{N}: n \geq l \rbrace \subseteq \lbrace n \in \mathbb{N}: t^{\mathbf{S}}(a(i,j)_n) = a(k_1 i + k_2 m,k_1 j)_n \rbrace.
	\end{equation*}
	Since $\mathfrak{U}$ is non-principal, it must contain all cofinite subsets of $\mathbb{N}$. Therefore, $t^{\mathbf{S}^*}(f(i,j)) = f(k_1 i + k_2 m,k_1 j)$. Since $f$ is an embedding, $f(t^{\mathbf{S}_m^\omega}(i,j)) = f(k_1 i + k_2 m,k_1 j)$ and $t^{\mathbf{S}_m^\omega}(i,j) = (k_1 i + k_2 m,k_1 j)$.  This concludes the proof.  
\end{proof}

\section{Conuclei on hoops}\label{sec:conuclei_on_hoops}
An \emph{interior operator} on a poset $\mathbf{P}$ is a map $\delta \colon P \to P$ such that 
\begin{equation*}
	\delta(a) \leq a, \quad a \leq b \Rightarrow \delta(a) \leq \delta(b) \quad \text{and} \quad \delta(a) = \delta(\delta(a))
\end{equation*}
for all $a,b \in P$. Interior operators are determined by their image $\delta(P) = \lbrace a \in P: \delta(a) = a \rbrace$, because $\delta(a) = \max\lbrace b \in \delta(P): b \leq a \rbrace$ for all $a \in P$.

A $\emph{conucleus}$ on a (bounded) hoop $\mathbf{A}$ is an interior operator $\delta$ on $\mathbf{A}$ that satisfies 
\begin{equation*}
	\delta(a) \cdot \delta(b) \leq \delta(a \cdot b) \quad \text{and} \quad \delta(a) \cdot \delta(\top) = \delta(a)
\end{equation*}
for all $a,b \in A$. The following characterization is a consequence of the definition.

\begin{proposition}\label{prop:conuclei_characterization}
	Let $\mathbf{A}$ be a hoop and $\delta\colon A \to A$. The following conditions are equivalent:
	\begin{enumerate}[label=\rm{(\roman*)}]
		\item $\delta$ is a conucleus on $\mathbf{A}$;
		\item $\delta$ is an interior operator such that, for all $a,b \in A$,
		\begin{equation*}
			\delta(\delta(a) \cdot \delta(b)) = \delta(a) \cdot \delta(b) \ \text{and} \ \delta(a) \cdot \delta(\top) = \delta(a).
		\end{equation*}
	\end{enumerate}	
\end{proposition}

We continue with some basic algebraic properties of conuclei on hoops. Some of them are well-known results for interior operators on partially ordered sets (see, for example, \cite{PRIESTLEY2002}). 
\begin{lemma}\label{lemma:conuclei_properties}
	Let $\mathbf{A}$ be a hoop and $\delta$ a conucleus on $\mathbf{A}$. For all $a, b \in A$ the following hold:
	\begin{enumerate}[label=\rm{(\arabic*)}]
		\item $\delta(a \land b) = \delta(\delta(a) \land b) = \delta(\delta(a) \land \delta(b))$.
		\item If $B \subseteq A$ and $\bigvee \delta(B)$ exists, then $\delta(\bigvee \delta(B)) = \bigvee \delta(B)$.
		\item $\delta(a) = \max\lbrace b \in \delta(A): b \leq a \rbrace$. 
		\item If $\mathbf{A}$ has least element $\bot$, then $\delta(\bot) = \bot$.
		\item $\delta(\delta(a) \cdot \delta(b)) = \delta(a) \cdot \delta(b)$.
		\item $\delta(a \imp b) \leq \delta(a) \imp \delta(b)$.
		\item $\delta(\top)^2 = \delta(\top)$.
		\item $\delta(a) \leq \delta(\top) \cdot a$.
	\end{enumerate}
\end{lemma}
\begin{proof}
	Items (1) to (4) are well-known properties of interior operators and their proofs will be omitted. Item (5) is a consequence of Proposition~\ref{prop:conuclei_characterization}. To prove item (6) observe that
	\begin{equation*}
		\begin{aligned}
			(a \imp b) \cdot a \leq b & \Rightarrow \delta((a\imp b) \cdot a) \leq \delta(b) \\
			& \Rightarrow \delta(a \imp b) \cdot \delta(a) \leq \delta(b) \\
			& \Rightarrow  \delta(a \imp b) \leq \delta(a) \imp \delta(b).
		\end{aligned}
	\end{equation*}
	
	Item (7) is straightforward. The same occurs for item (8), since $\delta(a) = \delta(a) \cdot \delta(\top) \leq a \cdot \delta(\top)$.  
\end{proof}

In the following examples, let $\mathbf{A}$ be an arbitrary hoop. Recall that a nonempty subset $F$ of $A$ is a \emph{filter} on $\mathbf{A}$ if, for all $a,b \in A$:
\begin{enumerate}[label=-]
	\item $a \in F$ and $a \leq b$ implies $b \in F$;
	\item $a,b \in F$ implies $a \cdot b \in F$.
\end{enumerate}
The set of all filters ordered under inclusion forms a lattice that is isomorphic to the lattice of congruences on $\mathbf{A}$ (see, for example, \cite[Theorem~2.14]{FERREIRIM1992}).	
\begin{example}\label{ex:conuclei}
	Let $a_0$ be a fixed element of $\mathbf{A}$ such that $a_0^2 = a_0$. If we define  
	\begin{equation*}
		\delta_{a_0}(a) \coloneq a \cdot a_0 
	\end{equation*}
	for all $a \in A$, then $\delta_{a_0}$ is a conucleus on $\mathbf{A}$. Since $\mathbf{A}$ is a hoop, $a \cdot a_0 = a \land a_0$ for all $a \in A$ and $\delta(A) = \lbrace a \in A: a \leq a_0 \rbrace$.
	
	Assume that $\mathbf{A} \vDash x^n \approx x^{n+1}$ for some $n \in \mathbb{N}$. Then,
	\begin{equation*}
		\delta_n \coloneq (x^{n})^\mathbf{A}
	\end{equation*}
	defines a conucleus on $\mathbf{A}$. Additionally, $\delta_n(A) = \lbrace a \in A: a^2 = a \rbrace$.
	
	If $\mathbf{A}$ has least element $\bot$, then another example of a conucleus is the $\square$ operator, defined as
	\begin{equation*}
		\square a \coloneq
		\begin{cases}
			\bot & \text{if $a \neq \top$}; \\
			\top & \text{if $a = \top$};
		\end{cases}
	\end{equation*}
	for all $a \in A$. Clearly, $\square(A) = \lbrace \bot,\top \rbrace$. The $\square$ operator can be easily generalized. If $F$ is a filter on $\mathbf{A}$ and we define
	\begin{equation*}
		\delta_F(a) \coloneq
		\begin{cases}
			\bot & \text{if $a \notin F$}; \\
			a & \text{if $a \in F$},
		\end{cases}
	\end{equation*}
	for all $a \in A$, then $\delta_F$ is a conucleus on $\mathbf{A}$. Also, $\delta_F(A) = \lbrace \bot \rbrace \cup F$.
\end{example}

If $\mathbf{A} \coloneq \langle A, \cdot, \imp, \top \rangle$ is a hoop and $\delta$ a conucleus on $\mathbf{A}$, we define
\begin{equation*}
	\mathbf{A}_\delta \coloneq \langle \delta(A), \cdot, \imp_\delta, \delta(\top) \rangle
\end{equation*}
where $a \imp_\delta b \coloneq \delta(a \imp b)$ for all $a,b \in \delta(A)$. As a consequence of Lemma~\ref{lemma:conuclei_properties}, $\mathbf{A}_\delta$ is well defined. If $\mathbf{A}$ is a bounded hoop, we define  
\begin{equation*}
	\mathbf{A}_\delta \coloneq \langle \delta(A), \cdot, \imp_\delta,\bot,\delta(\top) \rangle
\end{equation*}
The algebra $\mathbf{A}_\delta$ is known as the \emph{conuclear image} of $\mathbf{A}$ relative to $\delta$. The next result summarizes some properties of this structure.
\begin{proposition}
	Let $\mathbf{A}$ be a hoop and $\delta$ a conucleus on $\mathbf{A}$. Then:
	\begin{enumerate}[label=\rm{(\arabic*)}]
		\item $\langle \delta(A),\cdot,\delta(\top) \rangle$ is a commutative monoid.
		\item $a \imp_\delta a = \delta(\top)$ for all $a \in \delta(A)$.
		\item $(a \cdot b) \imp_\delta c = a \imp_\delta (b \imp_\delta c)$ for all $a,b,c \in \delta(A)$.
		\item $\mathbf{A}_\delta$ is a hoop if and only if $a \cdot \delta(a \imp b) = b \cdot \delta(b \imp a)$ for all $a,b \in \delta(A)$.
	\end{enumerate}
\end{proposition}
\begin{proof}
	We are going to restrict ourselves to the proof of item (3). If $a,b,c \in \delta(A)$, observe that $(a \cdot b) \imp_\delta c = \delta((a \cdot b) \imp c)$ and $a \imp_\delta (b \imp_\delta c) = \delta(a \imp \delta(b \imp c))$. On the one hand, we know that $\delta(b \imp c) \leq b \imp c$. Therefore,
	\begin{equation*}
		\delta(a \imp \delta(b \imp c)) \leq \delta(a \imp (b \imp c)) = \delta((a \cdot b) \imp c).
	\end{equation*}
	On the other hand, as a consequence of item (6) in Lemma~\ref{lemma:conuclei_properties}, $\delta(a \imp (b \imp c)) \leq a \imp \delta(b \imp c)$. Hence,
	\begin{equation*}
		\delta((a \cdot b) \imp c) = \delta(a \imp (b \imp c)) = \delta(\delta(a \imp (b \imp c))) \leq \delta(a \imp \delta(b \imp c)). 
	\end{equation*}  
\end{proof}

Observe that $\mathbf{A}_\delta$ is not necessarily a hoop. By the previous result, the reason for this is the identity $x \cdot (x \imp y) \approx y \cdot (y \imp x)$. To show a counterexample, consider the structure $\mathbf{S}_4$ defined in Example~\ref{ex:hoops}. 

\begin{proposition}
	There is a conucleus $\delta$ on $\mathbf{S}_4$ such that $(\mathbf{S}_4)_\delta$ is not a hoop.
\end{proposition}
\begin{proof}
	It can be easily shown that the set $B \coloneq \lbrace 0, \frac{1}{4}, \frac{1}{2}, 1 \rbrace$ is closed under multiplication. It follows that $\delta(x) \coloneq \max \lbrace b \in B: b \leq x \rbrace$ is a conucleus on $\mathbf{S}_4$. However, $\frac{1}{2} \cdot \left(\frac{1}{2} \imp_\delta \frac{1}{4} \right) = 0 \neq \frac{1}{4} = \frac{1}{4} \cdot \left(\frac{1}{4} \imp_\delta \frac{1}{2} \right)$.  
\end{proof}

Taking into account the previous proposition, we consider conuclei with an additional property. If $\mathbf{A}$ is a hoop, we say that $\delta\colon A \to A$ is a \emph{multiplicative conucleus} on $\mathbf{A}$, if $\delta$ is an interior operator on $\mathbf{A}$ that satisfies the equality $\delta(a \cdot b) = \delta(a) \cdot \delta(b)$ for all $a,b \in A$. It can be easily shown that all the conuclei defined in Example~\ref{ex:conuclei} are multiplicative conuclei. The next result shows that the conuclear image relative to a multiplicative conucleus is a hoop.

\begin{proposition}\label{prop:image_of_mult_conucleus_is_hoop}
	If $\mathbf{A}$ is a hoop and $\delta$ is a multiplicative conucleus on $\mathbf{A}$, then $\mathbf{A}_\delta$ is a hoop.
\end{proposition}
\begin{proof}
	We need to prove that $\mathbf{A}_\delta$ satisfies the identity $x \cdot (x \imp y) \approx y \cdot (y \imp x)$. If $a,b \in \delta(A)$, then $a \cdot (a \imp_\delta b) = \delta(a \cdot (a \imp b)) = \delta(b \cdot (b \imp a)) =  b \cdot (b \imp_\delta a)$.  
\end{proof}

The reciprocal of the previous result does not hold, as we are going to show with the next proposition.
\begin{proposition}\label{prop:not_multiplicative_hoop_conuclear_image}
	There exists a non-multiplicative conucleus $\delta$ on $\mathbf{S}_3$ such that $(\mathbf{S}_3)_\delta$ is a hoop. 
\end{proposition}
\begin{proof}
	Observe that $B \coloneq \lbrace 0, \frac{1}{3},1 \rbrace$ is closed under multiplication. Therefore, $\delta(x) \coloneq \max \lbrace b \in B: b \leq x \rbrace$ is a conucleus on $\mathbf{S}_3$.  To start, we are going to show that $(\mathbf{S}_3)_\delta$ is a hoop. By Proposition~\ref{prop:conuclei_characterization}, we want to prove that $a \cdot \delta(a \imp b) = b \cdot \delta(b \imp a)$ for all $a,b \in B$. If $a = 1$, $b = 1$ or $a = b$ the equality is immediate. On the other hand,
	\begin{equation*}
		\frac{1}{3} \cdot \delta\left(\frac{1}{3} \imp 0\right) = \frac{1}{3} \cdot \delta\left(\frac{2}{3}\right) = \left(\frac{1}{3}\right)^2 = 0 = 0 \cdot \delta\left(0 \imp \frac{1}{3}\right).	
	\end{equation*}
	To end, we need to show that $\delta$ is not multiplicative. Observe that
	\begin{equation*}
		\left(\delta\left(\frac{2}{3}\right)\right)^2 = \left(\frac{1}{3}\right)^2 = 0 \neq \frac{1}{3} = \delta\left(\frac{1}{3}\right) = \delta\left(\left(\frac{2}{3}\right)^2\right).
	\end{equation*}
	 
\end{proof}

The following result is analogous to the Glivenko property relative to a nucleus (see, for example, \cite{ZHAO2013}). The Glivenko property characterizes when the image of a nucleus on $\mathbf{A}$, which can be endowed with a hoop structure, is a homomorphic image of $\mathbf{A}$. Observe that, apart from the fact that a conuclear image is not necessarily a hoop, the notion of a homomorphism between a hoop and a conuclear image is well defined, because both structures share the same language. 

\begin{theorem}\label{thm:conuclei_glivenko}
	Let $\mathbf{A}$ be a hoop and $\delta$ a conucleus on $\mathbf{A}$. The following are equivalent:
	\begin{enumerate}[label=\rm{(\roman*)}]
		\item $\delta(a) = a \cdot \delta(\top)$ for all $a \in A$;
		\item $\delta$ is a homomorphism from $\mathbf{A}$ onto $\mathbf{A}_\delta$.
	\end{enumerate}
	
	If the equivalent conditions {\rm (i)} and {\rm (ii)} hold, then $\mathbf{A} \slash \ker (\delta) \cong \mathbf{A}_\delta$. Moreover, the filter corresponding to the congruence $\ker (\delta)$ is the set $\lbrace a \in A: \delta(\top) \leq a \rbrace$.
\end{theorem}
\begin{proof}
	Let us assume that $\delta(a) = a \cdot \delta(\top)$ for all $a \in A$. Given $a,b \in A$, we need to show that $\delta(a \cdot b) = \delta(a) \cdot \delta(b)$ and $\delta(a \imp b) = \delta(\delta(a) \imp \delta(b))$.
	
	The first equality is immediate. The inequality $\delta(a \imp b) \leq \delta(\delta(a) \imp \delta(b))$ is a straightforward consequence of the inequality $\delta(a \imp b) \leq \delta(a) \imp \delta(b)$ (see item (6) in Lemma~\ref{lemma:conuclei_properties}). On the other hand,
	\begin{equation*}
		\begin{aligned}
			\delta(\delta(a) \imp \delta(b)) \cdot a & = ((a \cdot \delta(\top)) \imp (b \cdot \delta(\top))) \cdot \delta(\top) \cdot a \\
			& \leq b \cdot \delta(\top) \\
			& \leq b.
		\end{aligned}
	\end{equation*} 	
	Hence, by the residuation property and the fact that $\delta(\top)^2 = \delta(\top)$,
	\begin{equation*}
		\begin{aligned}
			\delta(\delta(a) \imp \delta(b)) & = \delta(\delta(a) \imp \delta(b)) \cdot \delta(\top) \\
			& \leq (a \imp b) \cdot \delta(\top) \\
			& = \delta(a \imp b).
		\end{aligned}
	\end{equation*}
	
	Now assume that $\delta$ is a homomorphism from $\mathbf{A}$ onto $\mathbf{A}_\delta$. Given $a \in A$, we already know that $\delta(a) \leq a \cdot \delta(\top)$. On the other hand, using the hypothesis, we have that
	\begin{equation*}
		\begin{aligned}
			\delta(\top) & = \delta(\delta(a) \imp \delta(\delta(a))) \\
			& = \delta(a \imp \delta(a)) \\
			& \leq a \imp \delta(a). 
		\end{aligned}
	\end{equation*}
	
	If $\delta$ is a homomorphism, then it is clearly surjective and $\mathbf{A} \slash \ker(\delta) \cong \mathbf{A}_\delta$, where $\ker(\delta) = \lbrace (a,b) \in A^2: \delta(a) = \delta(b) \rbrace$. The filter corresponding to $\ker(\delta)$ is $\top \slash \ker(\delta) = \lbrace a \in A: \delta(a) = \delta(\top) \rbrace = \lbrace a \in A: \delta(\top) \leq a \rbrace$. 
	 
\end{proof}

\section{Conuclei defined by terms}\label{sec:conuclei_defined_by_terms}

Since terms can be interpreted on arbitrary structures of the same language, the notion of a conucleus defined by a term arises naturally. A term $t(x)$ in the language of (bounded) hoops is said to be a \emph{conucleus on a variety} of (bounded) hoops $\mathcal{V}$ if the interpretation $t^\mathbf{A}$ is a conucleus on $\mathbf{A}$ for all $\mathbf{A} \in \mathcal{V}$. We already showed in Example~\ref{ex:conuclei} that the identity term $x$ defines a conucleus on every hoop $\mathbf{A}$. In the context of bounded hoops, another example is the term $\bot$. Taking this into account, we say that a conucleus $t(x)$ on a variety of hoops (resp. bounded hoops) $\mathcal{V}$ is \emph{trivial} if $\mathcal{V} \vDash t(x) \approx x$ (resp. $\mathcal{V} \vDash t(x) \approx s(x)$ for some $s(x) \in \lbrace \bot,x \rbrace$). Observe that the condition of being a conucleus on a variety can be expressed with identities. Specifically, $t(x)$ is a conucleus on a variety of (bounded) hoops $\mathcal{V}$ if and only if $\mathcal{V}$ satisfies the identities $t(x) \imp x \approx \top$, $t(x \land y) \imp t(x) \approx \top$, $t(x) \approx t(t(x))$, $(t(x) \cdot t(y)) \imp t(x \cdot y) \approx \top$ and $t(x) \cdot t(\top) \approx t(x)$. 

In this section we provide some examples of terms that define conuclei on a variety of hoops or bounded hoops. The first example that will be given is in the context of the variety of Wajsberg hoops $\mathcal{WH}$. We are going to prove that the only term (up to equivalence) that defines a conucleus on the variety $\mathcal{WH}$ is the term $x$ (Proposition~\ref{prop:WH_conuclei}). It follows that the only terms (up to equivalence) that define a conucleus on the variety $\mathcal{MV}$ are the terms $x$ and $\bot$ (Corollary~\ref{cor:MV_conuclei}). Another example of a variety that only has trivial conuclei is the variety of cancellative hoops $\mathsf{V}(\mathbf{S}^\omega)$. The first example of a nontrivial conucleus on a variety of hoops will be given in the context of varieties of $k$-potent hoops. We end this section showing all the terms (up to equivalence) that define conuclei on the varieties $\mathsf{V}(\mathbf{S}_5)$ y $\mathsf{V}(\mathbf{S}_5^+)$. In both cases we obtain nontrivial conuclei. On the other hand, these examples show us that calculating all terms that define conuclei on a variety of Wajsberg hoops or MV-algebras can result in a nontrivial task. Taking this into account, we are going to restrict ourselves to the study of terms that define multiplicative conuclei on varieties of Wajsberg hoops, basic hoops and BL-algebras. This will be the main topic of the following sections.

We start with a general result that will be very useful for our examples. The following proposition relates the conuclei on a variety generated by an MV-algebra and the conuclei on the variety generated by the corresponding Wasjberg hoop with least element.
\begin{proposition}\label{prop:MV_conuclei_terms}
	Let $\mathbf{A}$ be a Wajsberg hoop with least element and let $t(x)$ be a term in the language $\lbrace \cdot,\imp,\bot,\top \rbrace$. The following are equivalent:
	\begin{enumerate}[label=\rm{(\roman*)}]
		\item $t(x)$ is a conucleus on $\mathsf{V}(\mathbf{A}^+)$;
		\item $\mathsf{V}(\mathbf{A}^+) \vDash t(x) \approx \bot$ or $\mathsf{V}(\mathbf{A}^+) \vDash t(x) \approx \overline{t}(x)$, where $\overline{t}(x)$ is a term in the language $\lbrace \cdot,\imp,\top \rbrace$ that is a conucleus on $\mathsf{V}(\mathbf{A})$.
	\end{enumerate}
\end{proposition}
\begin{proof}
	Assume first that $t(x)$ is a conucleus on $\mathsf{V}(\mathbf{A}^+)$. If $t(\top) \approx \bot$, then $\mathsf{V}(\mathbf{A}^+) \vDash t(x) \approx \bot$. Suppose now that $t(\top) \approx \top$. By Proposition~\ref{prop:MVterm_description}, there exists some term $\overline{t}(x)$ in the language $\lbrace \cdot,\imp,\top \rbrace$ such that $\mathcal{MV} \vDash t(x) \approx \overline{t}(x)$. Therefore, $t^{\mathbf{A}^+} = \overline{t}^{\mathbf{A}^+} = \overline{t}^\mathbf{A}$. Hence, $\overline{t}^\mathbf{A}$ is a conucleus on $\mathbf{A}$. 
	
	Let us prove the implication (ii) $\Rightarrow$ (i). If $\mathsf{V}(\mathbf{A}^+) \vDash t(x) \approx \bot$, then there is nothing to prove. Assume that $\mathsf{V}(\mathbf{A}^+) \vDash t(x) \approx \overline{t}(x)$, where $\overline{t}(x)$ is a term in the language $\lbrace \cdot,\imp,\top \rbrace$ that is a conucleus on $\mathsf{V}(\mathbf{A})$. Since $t^{\mathbf{A}^+} = \overline{t}^{\mathbf{A}^+} = \overline{t}^\mathbf{A}$, we obtain that $t^{\mathbf{A}^+}$ is a conucleus on $\mathbf{A}^+$. This concludes the proof.  
\end{proof}

We start with the study of conuclei on the variety $\mathcal{WH}$. Recall that $\mathcal{WH} = \mathsf{V}(\mathbf{S})$, where $\mathbf{S}$ is the totally ordered Wajsberg hoop defined in Example~\ref{ex:hoops}. 

\begin{proposition}\label{prop:WH_conuclei}
	$t(x)$ is a conucleus on $\mathcal{WH}$ if and only if it is trivial.
\end{proposition}	
\begin{proof}
	We are going to restrict the proof to the implication (i) $\Rightarrow$ (ii). Let $t(x)$ be a term such that $t^\mathbf{S}$ is a conucleus on $\mathbf{S}$. For all $m \in \mathbb{N}$, the following hold:
	\begin{enumerate}[label=-]
		\item $t^\mathbf{S}(\frac{m}{m+1}) = \frac{i}{m+1}$ for some $i \in \lbrace 0,\dots,m \rbrace$;
		\item $\frac{m}{m+2} < \frac{m}{m+1}$.
	\end{enumerate}	
	The first item is a consequence of the facts that $t^\mathbf{S} \leq x^\mathbf{S}$ and $\langle \frac{m}{m+1} \rangle^\mathbf{S} = S_{m+1}$, where $\langle \frac{m}{m+1} \rangle^\mathbf{S}$ is the subalgebra generated by $\frac{m}{m+1}$. The second item is immediate (recall that the hoop order on $\mathbf{S}$ is the usual order on $[0,1]$).
	
	Taking into consideration the previous remarks, we are going to show that, for each $m \in \mathbb{N}$,
	\begin{enumerate}[label=-]
		\item if $t^\mathbf{S}(\frac{m}{m+1}) = 0$, then $t^\mathbf{S}(\frac{m+1}{m+2}) = 0$;
		\item if $t^\mathbf{S}(\frac{m}{m+1}) = \frac{m}{m+1}$, then $t^\mathbf{S}(\frac{m+1}{m+2}) = \frac{m+1}{m+2}$.
	\end{enumerate}	
	First assume that $t^\mathbf{S}(\frac{m}{m+1}) = 0$. As a consequence of the previous remarks and the monotonicity of $t^\mathbf{S}$, $t^\mathbf{S}(\frac{i}{m+2}) = 0$ for all $i \in \lbrace 0,\dots,m \rbrace$. If $t^\mathbf{S}(\frac{m+1}{m+2}) = \frac{i}{m+2}$ for some $0 < i \leq m$, then $\frac{i}{m+2} = t^\mathbf{S}(\frac{i}{m+2}) = 0$ and we reached a contradiction. Therefore, $t^\mathbf{S}(\frac{m+1}{m+2}) = 0$ or $t^\mathbf{S}(\frac{m+1}{m+2}) = \frac{m+1}{m+2}$. If $t^\mathbf{S}(\frac{m+1}{m+2}) = \frac{m+1}{m+2}$ then, since $\frac{m}{m+2} = \left( \frac{m+1}{m+2} \right)^2$ and the image of $t^\mathbf{S}$ is closed under multiplication, $\frac{m}{m+2} = t^\mathbf{S}(\frac{m}{m+2}) = 0$, which is a contradiction. Hence, $t^\mathbf{S}(\frac{m+1}{m+2}) = 0$ and we proved the first item. Now assume that $t^\mathbf{S}(\frac{m}{m+1}) = \frac{m}{m+1}$. If $t^\mathbf{S}(\frac{m+1}{m+2}) < \frac{m+1}{m+2}$, taking into consideration the previous remarks , $t^\mathbf{S}(\frac{m+1}{m+2}) \leq \frac{m}{m+2} < \frac{m}{m+1} = t^\mathbf{S}(\frac{m}{m+1})$. This contradicts the monotonicity of $t^\mathbf{S}$, because $\frac{m}{m+1} < \frac{m+1}{m+2}$.
	
	We now continue with the proof of the main result. Observe that $t^\mathbf{S}(\frac{1}{2}) \in \lbrace 0, \frac{1}{2} \rbrace$. If $t^\mathbf{S}(\frac{1}{2}) = 0$, as a consequence of the previous remarks, $t^\mathbf{S}(\frac{m}{m+1}) = 0$ for all $m \in \mathbb{N}$. Since the sequence $(\frac{m}{m+1})_{m \in \mathbb{N}}$ is strictly increasing and converges to $1$, we conclude that
	\begin{equation*}
		t^\mathbf{A}(a) =
		\begin{cases}
			0 & \text{if $a \neq 1$;} \\
			1 & \text{if $a = 1$.}
		\end{cases}
	\end{equation*}
	This is a contradiction, because $t^\mathbf{S}$ is a McNaughton function and hence, continuous. This contradiction was obtained from the assumption that $t^\mathbf{S}(\frac{1}{2}) = 0$. Therefore, $t^\mathbf{S}(\frac{1}{2}) = \frac{1}{2}$ and by the previous remarks $t^\mathbf{S}(\frac{m}{m+1}) = \frac{m}{m+1}$ for all $m \in \mathbb{N}$. The sequence  $(\frac{m}{m+1})_{m \in \mathbb{N}}$ is the sequence of coatoms of each $\mathbf{S}_{m+1}$. Since every $S_{m+1}$ is equal to the set of powers of its coatom and the image of a conucleus is closed under multiplication, $\mathbb{Q} \cap [0,1] = \cup_{m \in \mathbb{N}} S_{m+1}$ is a subset of $t^\mathbf{S}(S)$ which is dense in $[0,1]$. Since $t^\mathbf{S}$ is continuous, $\mathbf{S} \vDash t(x) \approx x$.  
\end{proof}

The next result is a consequence of the previous proposition and Proposition~\ref{prop:MV_conuclei_terms}. Recall that $\mathcal{MV}$ is generated, as a variety, by $\mathbf{S}^+$.

\begin{corollary}\label{cor:MV_conuclei}
	$t(x)$ is a conucleus on $\mathcal{MV}$ if and only if it is trivial.
\end{corollary}

Recall that a hoop $\mathbf{A}$ is cancellative if it satisfies the identity $x \imp (x \cdot y) \approx y$. It is known that the variety of cancellative hoops is generated by $\mathbf{S}^\omega$ (see, for example, \cite{FERREIRIM1992}).

\begin{proposition}\label{prop:cancellative_conuclei}
	$t(x)$ is a conucleus on $\mathsf{V}(\mathbf{S}^\omega)$ if and only if it is trivial.
\end{proposition}
\begin{proof}
	As a consequence of Proposition~\ref{prop:to_wajsberg_is_semi_cancellative}, $\langle a \rangle^{\mathbf{S}^\omega} \cong \mathbf{S}^\omega$ for all $a \in S^\omega \backslash \lbrace \top \rbrace$. If $c$ is the coatom of $\mathbf{S}^\omega$, then $t^{\mathbf{S}^\omega}(c) = c^n$ for some $n \geq 1$. Particularly, since the image of $t^{\mathbf{S}^\omega}$ is closed under multiplication, $t^{\mathbf{S}^\omega}(c^{nm}) = c^{nm}$ for all $m \geq 0$. Therefore, $\mathbf{S}^\omega \cong \langle c^n \rangle^{\mathbf{S}^\omega} \vDash t(x) \approx x$.  
\end{proof}

We have given several examples of varieties that only have trivial conuclei. This is not always the case, as we will show next. Recall that a hoop is \emph{$k$-potent} ($k \in \mathbb{N}$) if it satisfies the identity $x^k \approx x^{k+1}$. Let $\mathcal{H}_k$ denote the variety of $k$-potent hoops.
\begin{proposition}
	If $k \geq 2$, then $x^k$ is a nontrivial conucleus on $\mathcal{H}_k$.
\end{proposition}
\begin{proof}
	As a consequence of Example~\ref{ex:conuclei}, $x^k$ is a conucleus on $\mathcal{H}_k$. Moreover, $\mathbf{S}_k \in \mathcal{H}_k$ and $a^k = \square a$ for all $a \in S_k$. Hence, $x^k$ is nontrivial.
\end{proof}

Before we continue with another example of a variety with nontrivial conuclei, we list some important properties of $\mathbf{S}_n$.
\begin{lemma}\label{lemma:Sn_conuclei_properties}
	If $n \in \mathbb{N}$, the following hold:
	\begin{enumerate}[label=\rm{(\arabic*)}]
		\item A function $\delta \colon S_n \to S_n$ can be represented with a term in the language $\lbrace \cdot,\imp,\top \rbrace$ if and only if $\delta(a) \in \langle a \rangle^{\mathbf{S}_n}$ for all $a \in S_n$, where $\langle a \rangle^{\mathbf{S}_n}$ is the subalgebra generated by $a$.
		\item If $n$ is a prime number, then a function $\delta \colon S_n \to S_n$ is representable with a term in the language $\lbrace \cdot,\imp,\top \rbrace$ if and only if $\delta(0) \in \lbrace 0,1 \rbrace$ and $\delta(1) = 1$.
		\item For each $B \subseteq S_n$, $B$ is the image of a conucleus on $\mathbf{S}_n$ if and only if
		\begin{enumerate}[label=-]
			\item $B$ is closed under multiplication and
			\item $B = \lbrace 0 \rbrace$ or $1 \in B$.
		\end{enumerate}
	\end{enumerate}
\end{lemma}
\begin{proof}
	Items (1) and (2) are well-known facts about Wajsberg hoops (see, for example, \cite[Theorem 3.3, Corollary 3.5]{AGLIANO2002}).
	
	Let us prove item (3). Let $B$ be a subset of $S_n$ and assume that there exists some conucleus $\delta\colon S_n \to S_n$ such that $\delta(S_n) = B$. By Proposition~\ref{prop:conuclei_characterization}, $B$ is closed under multiplication and $\delta(1) \cdot a = a$ for all $a \in B$. Since $\delta(1)$ is idempotent, necessarily $\delta(1) \in \lbrace 0,1 \rbrace$. If $\delta(1) = 0$, then $\delta(a) = 0$ for all $a \in S_n$ and $B = \lbrace 0 \rbrace$. If $\delta(1) \neq 0$, then $\delta(1) = 1 \in B$. This concludes the proof of the first implication. Now assume that $B$ is closed under multiplication and $B = \lbrace 0 \rbrace$ or $1 \in B$. In this case, we define $\delta(a) \coloneq \max \lbrace b \in B: b \leq a \rbrace$ for all $a \in A$. Clearly $\delta$ is an interior operator, $\delta(S_n) = B$ and $\delta(a) \cdot \delta(1) = \delta(a)$ for all $a \in S_n$. Moreover, since $B$ is closed under multiplication, $\delta(a) \cdot \delta(b) \leq \delta(a \cdot b)$ for all $a,b \in S_n$. Hence, $\delta$ is a conucleus.  
\end{proof}

Now consider the totally ordered Wajsberg hoop $\mathbf{S}_5$. Taking into consideration the previous lemma, a simple computation yields the following images of conuclei:
\begin{equation*}
	\textstyle \lbrace 0 \rbrace, \lbrace 0,1 \rbrace, \lbrace 0,\frac{1}{5}, 1 \rbrace, \lbrace 0,\frac{2}{5}, 1 \rbrace, \lbrace 0,\frac{1}{5},\frac{2}{5}, 1 \rbrace, \lbrace 0,\frac{1}{5},\frac{3}{5}, 1 \rbrace, \lbrace 0,\frac{1}{5},\frac{2}{5},\frac{3}{5}, 1 \rbrace, S_5.
\end{equation*}		
Clearly, the conucleus corresponding to the set $\lbrace 0 \rbrace$ can not be represented with a term in the language $\lbrace \cdot,\imp,\top \rbrace$. If we define $s(x) \coloneq x \imp x^2$, the conuclei corresponding to each image different from $\lbrace 0 \rbrace$ are, in order of appearance, the interpretation of the terms
\begin{equation*}
	\begin{aligned}
		\delta_1(x) & = x^5, \\
		\delta_2(x) & = (s(x) \imp x^3) \cdot (s(x)^3 \imp x^3), \\
		\delta_3(x) & = (x \imp s(x)^2) \cdot ((x \imp x^3) \imp x^5), \\
		\delta_4(x) & = (((s(x) \imp x^3) \imp s(x)^2) \imp ((x \imp x^3) \cdot (x^2 \imp x^3))) \imp x^5, \\
		\delta_5(x) & = s(x) \imp x^4, \\
		\delta_6(x) & = x \cdot (x^3 \imp x^4) \text{ and } \\
		\delta_7(x) & = x.
	\end{aligned}
\end{equation*}
Therefore, a term $t(x)$ is a conucleus on $\mathsf{V}(\mathbf{S}_5)$ if and only if $\mathsf{V}(\mathbf{S}_5) \vDash t(x) \approx \delta_i(x)$ for some $i \in \lbrace 1,\dots,7 \rbrace$. Furthermore, a term $t(x)$ is a conucleus on $\mathsf{V}(\mathbf{S}_5^+)$ if and only if $\mathsf{V}(\mathbf{S}_5^+) \vDash t(x) \approx \bot$ or $\mathsf{V}(\mathbf{S}_5^+) \vDash t(x) \approx \delta_i(x)$ for some $i \in \lbrace 1,\dots,7 \rbrace$.

We have provided several examples of terms that define nontrivial conuclei on the varieties $\mathsf{V}(\mathbf{S}_5)$ and $\mathsf{V}(\mathbf{S}_5^+)$. These examples show that calculating all conuclei on some variety of Wajsberg hoops or MV-algebras, even if it is generated by a finite algebra, can be a nontrivial task. This is a consequence of the characterization of conuclear images given by item (3) in Lemma~\ref{lemma:Sn_conuclei_properties}, which essentially tells us that the problem of calculating conuclei on some $\mathbf{S}_n$ is equivalent to the problem of finding all subsets closed under multiplication. Taking this into consideration, we are going to turn our focus to terms that define multiplicative conuclei on varieties of Wajsberg hoops. In this case, an explicit description of these terms can be provided. This will be the main topic of the following section.

\section{Multiplicative conuclei on varieties of Wajsberg hoops}\label{sec:multiplicative_Wajsberg_conuclei}

Recall that $\delta$ is a multiplicative conucleus on a (bounded) hoop $\mathbf{A}$ if and only if it is an interior operator that satisfies $\delta(a) \cdot \delta(b) = \delta(a \cdot b)$ for all $a,b \in A$. The definition of a term that defines a multiplicative conucleus on a variety follows naturally: a term $t(x)$ in the language of (bounded) hoops is a \emph{multiplicative conucleus} on a variety of (bounded) hoops $\mathcal{V}$ if $t^\mathbf{A}$ is a multiplicative conucleus on $\mathbf{A}$ for all $\mathbf{A} \in \mathcal{V}$. As is the case with conuclei on varieties, observe that being a multiplicative conucleus on a variety can also be expressed with identities.

The main result of this section is a characterization of all terms that define multiplicative conuclei on proper subvarieties of $\mathcal{WH}$ (Theorem~\ref{thm:proper_WH_multiplicative_conuclei}). As a consequence of this result, we can prove that there exists an effective procedure to calculate all multiplicative conuclei on a given proper subvariety of $\mathcal{WH}$. This also shows that the number of terms that define a multiplicative conucleus on a variety of Wajsberg hoops is, up to equivalence, finite. We end this section with an example.

The first positive result of this section is a description of all multiplicative conuclei on a finite totally ordered Wajsberg hoop.

\begin{proposition}\label{prop:Sn_multiplicative_conuclei}
	Let $n \in \mathbb{N}$ and $\delta\colon S_n \to S_n$. The following are equivalent:
	\begin{enumerate}[label=\rm{(\roman*)}]
		\item $\delta$ is a multiplicative conucleus on $\mathbf{S}_n$; 
		\item $\delta \in \lbrace \bot^{\mathbf{S}_n^+},(x^n)^{\mathbf{S}_n},x^{\mathbf{S}_n} \rbrace$.
	\end{enumerate}
\end{proposition}
\begin{proof}
	We restrict ourselves to the implication (i) $\Rightarrow$ (ii). Assume that $\delta$ is a multiplicative conucleus on $\mathbf{S}_n$. Taking into consideration item (7) in Lemma~\ref{lemma:conuclei_properties}, we know that $\delta(1) \in \lbrace 0,1 \rbrace$. If $\delta(1) = 0$, then clearly $\delta = \bot^{\mathbf{S}_n^+}$. Starting now assume that $\delta(1) = 1$. If $c$ is the coatom of $\mathbf{S}_n$, then $\delta(c) = c^m$ for some $m \in \lbrace 1,\dots,n \rbrace$ (recall that $\mathbf{S}_n = \langle c \rangle^{\mathbf{S}_n} = \lbrace c^k: k \in \lbrace 0,\dots,n \rbrace \rbrace$, where $\langle c \rangle^{\mathbf{S}_n}$ is the subalgebra generated by $c$). If $m = 1$, then $\delta(c) = c$ and $\delta(c^k) = c^k$ for all $k \in \lbrace 0,\dots,n \rbrace$. In this case $\delta = x^{\mathbf{S}_n}$. On the other hand, if $m > 1$, then $\delta(c^{m-1}) = c^m$ (because $c^m = \delta(c^m) \leq \delta(c^{m-1}) \leq \delta(c) = c^m$) and
	\begin{equation*}
		c^m = \delta(c^m) = \delta(c \cdot c^{m-1}) = \delta(c) \cdot \delta(c^{m-1}) = (c^m)^2.
	\end{equation*}	
	Since $c^m < 1$, necessarily $\delta(c) = c^m = 0$. Since $\delta$ is monotone, $\delta = (x^n)^{\mathbf{S}_n}$.  
\end{proof}

Observe that, in the previous proposition, the multiplicative conucleus $x^n$ is none other than the $\square$ operator on $\mathbf{S}_n$ (recall Example~\ref{ex:conuclei}). Before we continue with the next result, we consider some properties of $\mathbf{S}_n^\omega$.
\begin{lemma}\label{lemma:Smomega_properties}
	If $n \in \mathbb{N}$, then the following hold:
	\begin{enumerate}[label=\rm{(\arabic*)}]
		\item If $m \in \lbrace 0,\dots,n \rbrace$ and $k \in \mathbb{Z}$ are such that $(m,k)$, $(m,k-1) \in S_n^\omega$, then $(m,k) \cdot (n,-1) = (m,k-1)$.
		\item If $\delta$ is a multiplicative conucleus on $\mathbf{S}_n^\omega$, then $\delta(a) \in \lbrace (0,0), a \rbrace$ for all $a \in S_n^\omega$. Particularly, $\delta(S_n^\omega) \backslash \lbrace (0,0) \rbrace$ is an up-set, that is, if $\delta(a) = a \neq (0,0)$ and $a \leq b$, then $\delta(b) = b \neq (0,0)$, for all $a,b \in S_n^\omega$.
	\end{enumerate}
\end{lemma}
\begin{proof}
	The first item is a well-known property of $\mathbf{S}_n^\omega$ and its proof will be omitted. To prove the second fact, let $\delta$ be a multiplicative conucleus on $\mathbf{S}_n^\omega$. Let $a \in S_n^\omega$ and assume that $\delta(a) < a$. If $(0,0) < \delta(a)$, let $\delta(a)^-$ and $\delta(a)^+$ be the predecessor and successor of $\delta(a)$, respectively. Observe that $\delta(a)^+ \in [\delta(a),a]$. Therefore $\delta(\delta(a)^+) = \delta(a)$. Furthermore, as a consequence of item (1), $\delta(\delta(a)^+ \cdot (n,-1)) = \delta(\delta(a)) = \delta(a)$ and 
	\begin{equation*}
		\delta(\delta(a)^+) \cdot \delta(n,-1) = \delta(a) \cdot \delta(n,-1) \leq \delta(a) \cdot (n,-1) = \delta(a)^-.
	\end{equation*}
	Hence, $\delta(\delta(a)^+ \cdot (n,-1)) \neq \delta(\delta(a)^+) \cdot \delta(n,-1)$ and we reached a contradiction, that was achieved by assuming that $(0,0) < \delta(a)$.  
\end{proof}

For the next result recall that, if $F$ is a filter on a hoop with least element $\bot$, then $\delta_F$ is a multiplicative conucleus, where
\begin{equation*}
	\delta_F(x) = 
	\begin{cases}
		\bot & \text{if $x \notin F$}; \\
		a & \text{if $x \in F$}.
	\end{cases}
\end{equation*}
Observe that, for each $n \in \mathbb{N}$, the set $F = \lbrace (n,j): j \leq 0 \rbrace$ is a filter on $\mathbf{S}_n^\omega$ and $\delta_F$ is a multiplicative conucleus on $\mathbf{S}_n^\omega$. As will be shown next, $\delta_F$ can be represented with a term.

\begin{lemma}\label{lemma:Snomega_nontrivial_filter_conucleus}
	Let $n \in \mathbb{N}$ and $F \coloneq \lbrace (n,j): j \leq 0 \rbrace$. Then,
	\begin{equation*}
		\delta_F = (x \cdot ((x^{n+1} \imp x^{2(n+1)}) \imp x^{n+1}))^{\mathbf{S}_n^\omega}.
	\end{equation*}
\end{lemma}
\begin{proof}
	The proof is an immediate consequence of the fact that $a^{n+1} = (0,0)$ for all $a \in S_n^\omega \backslash F$ and $a^{n+1} \imp a^{2(n+1)} = a^{n+1}$ for all $a \in F$.  
\end{proof}

Thanks to the previous results, we can describe all multiplicative conuclei on $\mathbf{S}_n^\omega$ for all $n \in \mathbb{N}$. 

\begin{proposition}\label{prop:Snomega_multiplicative_conuclei}
	Let $n \in \mathbb{N}$ and $\delta\colon S_n^\omega \to S_n^\omega$. The following are equivalent:
	\begin{enumerate}[label=\rm{(\roman*)}]
		\item $\delta$ is a multiplicative conucleus on $\mathbf{S}_n^\omega$; 
		\item $\delta(x) \in \lbrace \bot^{(\mathbf{S}_n^\omega)^+}, \square, (x \cdot ((x^{n+1} \imp x^{2(n+1)}) \imp x^{n+1}))^{\mathbf{S}_n^\omega}, x^{\mathbf{S}_n^\omega} \rbrace$.
	\end{enumerate}
\end{proposition}
\begin{proof}
	We are going to restrict the proof to the implication (i) $\Rightarrow$ (ii). Assume that $\delta$ is a multiplicative conucleus on $\mathbf{S}_n^\omega$. If $\delta((n,0)) = (0,0)$, then $\delta = \bot^{(\mathbf{S}_n^\omega)^+}$. If $\delta((n,0)) = (n,0)$, we have to consider different cases. Before we continue, we define $c \coloneq (n,-1)$.
	
	We first consider the case where $b \coloneq \min \lbrace \delta(S_n^\omega) \backslash \lbrace (0,0) \rbrace \rbrace$ exists. If $b = (n,0)$, then $\delta(c) = (0,0)$ and $\delta = \square$. Assume that $b < (n,0)$. Taking into consideration Lemma~\ref{lemma:Smomega_properties}, we conclude that $\delta(c) = c$. Since $b > (0,0)$, the element $b$ has a predecessor which will be denoted by $b^-$. Observe that $\delta(b^-) = (0,0)$ and
	\begin{equation*}
		b^- = b \cdot c = \delta(b) \cdot \delta(c) = \delta(b \cdot c) = \delta(b^-) = (0,0).
	\end{equation*}	
	Hence $b^- = (0,0)$, that is, $b = (0,1)$ is the atom of the structure. In this case, since $\delta(S_n^\omega) \backslash \lbrace (0,0) \rbrace$ is an up-set, $\delta = x^{\mathbf{S}_n^\omega}$.
	
	We now consider the case where $\min \lbrace \delta(S_n^\omega) \backslash \lbrace (0,0) \rbrace \rbrace$ does not exist. If we define $M \coloneq \lbrace 0 < k \leq n: (k,0) \in \delta(S_n^\omega) \backslash \lbrace (0,0) \rbrace \rbrace$, then $M \neq \emptyset$ and $m \coloneq \min M > 0$. Moreover, the following holds:
	\begin{enumerate}[label=-]
		\item $\delta(m,j) = (m,j)$ for all $j \leq 0$: if there exists some $j < 0$ such that $\delta(m,j) = (0,0)$, then $l \coloneq \min \lbrace k \leq 0: \delta(m,k) = (m,k) \rbrace$ exists and we obtain that $(m,l) = \min \lbrace \delta(S_n^\omega) \backslash \lbrace (0,0) \rbrace \rbrace$ (this is a contradiction). 
		\item $\delta(m-1,j) = (0,0)$ for all $j \geq 0$: if there exists some $j > 0$ such that $\delta(m-1,j) = (m-1,j)$, then $l \coloneq \min\lbrace k > 0: \delta(m-1,k) = (m-1,k) \rbrace$ exists and $(m-1,l) = \min \lbrace \delta(S_n^\omega) \backslash \lbrace (0,0) \rbrace \rbrace$ (a contradiction, again).
	\end{enumerate}		
	In this case $\delta = \delta_m$, where
	\begin{equation*}
		\delta_m(a) \coloneq
		\begin{cases}
			(0,0) & \text{if $a \in [(0,0),(m-1,0)] \cup \lbrace (m-1,j): j \geq 0 \rbrace$}; \\
			a & \text{if $a \in \lbrace (m,j): j \leq 0 \rbrace \cup [(m,0),(n,0)]$};
		\end{cases}
	\end{equation*}
	for each $a \in S_n^\omega$. Observe that $\delta_m$ is an interior operator with 
	\begin{equation*}
		\delta_m(S_n^\omega) = \lbrace (0,0) \rbrace \cup \lbrace (m,j): j \leq 0 \rbrace \cup [(m,0),(n,0)]
	\end{equation*}		
	By Proposition~\ref{prop:conuclei_characterization}, $\delta_m$ is a conucleus if and only if $\delta_m(S_n^\omega)$ is closed under multiplication. Moreover, as we are going to show next, $\delta_m(S_n^\omega)$ is closed under multiplication if and only if $n = m$. To prove this, we define
	\begin{equation*}
		l \coloneq 
		\begin{cases}
			\frac{n}{2} & \text{if $n$ even;} \\
			\frac{n-1}{2} & \text{otherwise.}
		\end{cases}
	\end{equation*}
	If $m \leq l$, then $(l,0) \cdot (n-l,1) = (0,1) \notin \delta(S_n^\omega)$. On the other hand, if $l < m < n$, $(m,0)^2 = (2m-n,0) \notin \delta(S_n^\omega)$. Therefore $n = m$ and $\delta = \delta_n$. By the definition of $\delta_n$ and Lemma~\ref{lemma:Snomega_nontrivial_filter_conucleus}, $\delta = (x \cdot ((x^{n+1} \imp x^{2(n+1)}) \imp x^{n+1}))^{\mathbf{S}_n^\omega}$.  
\end{proof}

Before we proceed, we need to discard the multiplicative conuclei on $\mathbf{S}_n^\omega$ that can not be represented with a term in the language of hoops.

\begin{proposition}
	If $n \in \mathbb{N}$ and $\delta \in \lbrace \bot^{(\mathbf{S}_n^\omega)^+},\square \rbrace$, then there does not exist a term $t(x)$ in the language $\lbrace \cdot,\imp,\top \rbrace$ such that $\delta = t^{\mathbf{S}_n^\omega}$.
\end{proposition}
\begin{proof}
	The proof for $\delta = \bot^{(\mathbf{S}_n^\omega)^+}$ is immediate, since $t(\top) \approx \top$ for each term $t(x)$ in the language of hoops. Now consider $\delta = \square$. If $c$ is the coatom of $\mathbf{S}_n^\omega$, then $\langle c \rangle^{\mathbf{S}_n^\omega} = \lbrace (n,j): j \leq 0 \rbrace$, where $\langle c \rangle^{\mathbf{S}_n^\omega}$ denotes the subalgebra generated by $c$. If $\delta$ can be represented with a term $t(x)$ in the language of hoops, then $(0,0) = \square c = t^{\mathbf{S}_n^\omega}(c) \in \langle c \rangle^{\mathbf{S}_n^\omega}$ and we reached a contradiction.  
\end{proof}

We now proceed with one of the main results of this section. The following theorem characterizes all multiplicative conuclei on an arbitrary proper subvariety of $\mathcal{WH}$. Before we continue, we ask the reader to recall Theorem~\ref{thm:proper_WH_varieties}.

\begin{theorem}\label{thm:proper_WH_multiplicative_conuclei}
	Let $(I,J,K)$ be a reduced triple and let $\mathcal{V}$ be the corresponding proper subvariety of $\mathcal{WH}$. If $t(x)$ is a term in the language $\lbrace \cdot,\imp,\top \rbrace$, the following are equivalent:
	\begin{enumerate}[label=\rm{(\roman*)}]
		\item $t(x)$ is a multiplicative conucleus on $\mathcal{V}$;
		\item there exist $I_1,I_2$ subsets of $I$, $J_1,J_2$ subsets of $J$ such that
		\begin{enumerate}[label=-]
			\item $I_1 \cup I_2 = I$, $I_1 \cap I_2 = \emptyset$,
			\item $(J_1,J_2) \in \lbrace (J,\emptyset), (\emptyset,J) \rbrace$,
			\item $\gcd(k,l) = 1$ for all $k \in I_1 \cup J_1$, $l \in I_2 \cup J_2$ 
		\end{enumerate}
		and
		\begin{enumerate}[label=-]
			\item $\mathbf{S}_i \vDash t(x) \approx x$ and $\mathbf{S}_j^\omega \vDash t(x) \approx x$ for all $i \in I_1$, $j \in J_1$,
			\item $\mathbf{S}_i \vDash t(x) \approx x^i$ for all $i \in I_2$,
			\item $\mathbf{S}_j^\omega \vDash t(x) \approx x \cdot ((x^{j+1} \imp x^{2(j+1)}) \imp x^{j+1})$ for all $j \in J_2$,
			\item if $K = \lbrace 0 \rbrace$, $\mathbf{S}^\omega \vDash t(x) \approx x$.
		\end{enumerate}
	\end{enumerate}
\end{theorem}
\begin{proof}
	Assume the hypothesis in (ii). By Proposition~\ref{prop:cancellative_conuclei}, Proposition~\ref{prop:Sn_multiplicative_conuclei} and Proposition~\ref{prop:Snomega_multiplicative_conuclei} we conclude that $t(x)$ is a multiplicative conucleus on every structure that generates $\mathcal{V}$. Hence, $t(x)$ is a multiplicative conucleus on $\mathcal{V}$.
	
	Now assume that $t(x)$ is a multiplicative conucleus on $\mathcal{V}$. Again, by Proposition~\ref{prop:cancellative_conuclei}, Proposition~\ref{prop:Sn_multiplicative_conuclei} and Proposition~\ref{prop:Snomega_multiplicative_conuclei},
	\begin{enumerate}[label=-]
		\item for all $i \in I$, $\mathbf{S}_i \vDash t(x) \approx t_i(x)$ for some $t_i(x) \in \lbrace x,x^i \rbrace$;
		\item for all $j \in J$, $\mathbf{S}_j^\omega \vDash t(x) \approx t_j(x)$ for some 
		\begin{equation*}
			t_j(x) \in \lbrace x, x \cdot ((x^{j+1} \imp x^{2(j+1)}) \imp x^{j+1}) \rbrace;
		\end{equation*}
		\item if $K = \lbrace 0 \rbrace$, $\mathbf{S}^\omega \vDash t(x) \approx x$.
	\end{enumerate}
	Next, we define
	\begin{enumerate}[label=-]
		\item $I_1 \coloneq \lbrace i \in I: t_i(x) = x \rbrace$;
		\item $I_2 \coloneq \lbrace i \in I: i > 1, t_i(x) = x^i \rbrace$;
		\item $J_1 \coloneq \lbrace j \in J: t_j(x) = x \rbrace$;
		\item $J_2 \coloneq \lbrace j \in J: t_j(x) = x \cdot ((x^{j+1} \imp x^{2(j+1)}) \imp x^{j+1}) \rbrace$.
	\end{enumerate} 		
	The equality $I_1 \cup I_2 = I$ is immediate. Suppose that $i \in I$ is such that $i \in I_1 \cap I_2$. Then $i \geq 2$ and $\mathbf{S}_i \vDash x \approx x^i$. This is a contradiction because, given $a \in S_i \backslash \lbrace 0,1\rbrace$, $a \neq 0 = a^i$. On the other hand, the equality $J_1 \cup J_2 = J$ is also immediate. We want to prove that $J_1 = \emptyset$ or $J_2 = \emptyset$. Assume that there exist $j_1,j_2$ elements of $J$ such that $j_1 \in J_1$ and $j_2 \in J_2$.  Then, $\mathbf{S}_1^\omega$ is isomorphic to a subalgebra of both $\mathbf{S}_{j_1}^\omega$ and $\mathbf{S}_{j_2}^\omega$. Therefore, $\mathbf{S}_1^\omega \vDash x \approx x \cdot ((x^{j_2+1} \imp x^{2(j_2+1)}) \imp x^{j_2+1})$. This is a contradiction, because if $a$ is the atom of $\mathbf{S}_1^{\omega}$, then
	\begin{equation*}
		a \neq (0,0) = a \cdot (0,0) = a \cdot ((a^{j_2+1} \imp a^{2(j_2+1)}) \imp a^{j_2+1}).
	\end{equation*}
	To end this proof, let $k \in I_1 \cup J_1$, $l \in I_2 \cup J_2$ and define $u \coloneq \gcd(k,l)$. Observe that $\mathbf{S}_u$ is isomorphic to a subalgebra of $\mathbf{A}$ and $\mathbf{B}$, where $\mathbf{A} \in \lbrace \mathbf{S}_k,\mathbf{S}_k^\omega \rbrace$ and $\mathbf{B} \in \lbrace \mathbf{S}_l,\mathbf{S}_l^\omega \rbrace$. Therefore, $\mathbf{S}_u \vDash x \approx s(x)$ for some $s(x) \in \lbrace x^l, x \cdot ((x^{l+1} \imp x^{2(l+1)}) \imp x^{l+1} \rbrace$. If $u \geq 2$ then $S_u \backslash \lbrace 0,1 \rbrace \neq \emptyset$. Particularly, if $a \in S_u \backslash \lbrace 0,1 \rbrace$, $a \neq s^{\mathbf{S}_u} (a)$, which is a contradiction. Hence, $u = 1$.  
\end{proof}

If $(I,J,K)$ is a reduced triple, $I_1,I_2 \subseteq I$ and $J_1,J_2 \subseteq J$ we say that $(I_1,I_2,J_1,J_2)$ is a \emph{multiplicative separation} of $(I,J,K)$ if the following conditions hold:
\begin{enumerate}[label=-]
	\item $I_1 \cup I_2 = I$, $I_1 \cap I_2 = \emptyset$;
	\item $(J_1,J_2) \in \lbrace (J,\emptyset), (\emptyset,J) \rbrace$;
	\item $\gcd(k,l) = 1$ for all $k \in I_1 \cup J_1$, $l \in I_2 \cup J_2$.
\end{enumerate}	

As a consequence of the previous theorem, for each term that defines a multiplicative conucleus on a proper subvariety of $\mathcal{WH}$ there must exist some multiplicative separation of the reduced triple corresponding to said subvariety. The reciprocal of this statement does also hold, which leads us to the second main result of this section. Not only does every multiplicative separation define a term that is a multiplicative conucleus on our subvariety, but this term can be effectively calculated.

\begin{theorem}\label{thm:multiplicative_separation_conucleus}
	Let $(I,J,K)$ be a reduced triple and let $\mathcal{V}$ be the corresponding proper subvariety of $\mathcal{WH}$. For each multiplicative separation $\Pi \coloneq (I_1,I_2,J_1,J_2)$ of $(I,J,K)$ there exists a term $t_\Pi(x)$, which can be effectively calculated, such that
	\begin{enumerate}[label=-]
		\item $\mathbf{S}_i \vDash t_\Pi(x) \approx x$ and $\mathbf{S}_j^\omega \vDash t_\Pi(x) \approx x$ for all $i \in I_1$, $j \in J_1$;
		\item $\mathbf{S}_i \vDash t_\Pi(x) \approx x^i$ for all $i \in I_2$;
		\item $\mathbf{S}_j^\omega \vDash t_\Pi(x) \approx x \cdot ((x^{j+1} \imp x^{2(j+1)}) \imp x^{j+1})$ for all $j \in J_2$;
		\item if $K = \lbrace 0 \rbrace$, $\mathbf{S}^\omega \vDash t_\Pi(x) \approx x$.
	\end{enumerate}
	
	Hence, $t_\Pi(x)$ is a multiplicative conucleus on $\mathcal{V}$.
\end{theorem}
\begin{proof}
	First we define
	\begin{equation*}
		T_x \coloneq [\cup_{l \in I_1 \cup J_1} S_l] \backslash \lbrace 0,1 \rbrace \quad \text{and} \quad T_0 \coloneq [\cup_{k \in I_2 \cup J_2} S_k] \backslash \lbrace 0,1 \rbrace.
	\end{equation*}
	To begin, let us prove that $T_x \cap T_0 = \emptyset$. Recall that $\mathbf{B}$ is a subalgebra of $\mathbf{S}_n$ ($n \in \mathbb{N}$) if and only if $\mathbf{B} = \mathbf{S}_0$ or $\mathbf{B} = \mathbf{S}_m$ for some $m \in \mathbb{N}$ that divides $n$ (see, for example, Proposition 1.4 in \cite[Section~3.3.1]{FERREIRIM1992}). If the intersection is not empty, then there exist $l \in I_1 \cup J_1$, $k \in I_2 \cup J_2$ and $a \in \mathbb{Q}$ such that $a \in (S_l \cap S_k) \backslash \lbrace 0,1 \rbrace$. Therefore, $\langle a \rangle^\mathbf{S}$ is not trivial and $\langle a \rangle^{\mathbf{S}_l} = \langle a \rangle^{\mathbf{S}} = \langle a \rangle^{\mathbf{S}_k}$, where $\langle a \rangle^\mathbf{A}$ is the subalgebra generated by $a$. Hence, $\langle a \rangle^\mathbf{S} = \mathbf{S}_u$ for some $u \in \mathbb{N}$ that divides $l$ and $k$. Since $a \notin \lbrace 0,1 \rbrace$, $2 \leq u \leq \text{mcd}(l,k) = 1$ and we reached a contradiction. 
	
	For each $n \in \mathbb{N}_0$ and $q \in \mathsf{Far}_n$ let $q^-$ and $q^+$ denote the predecessor and successor (if they exist) of $q$ in $\mathsf{Far}_n$, respectively. By Proposition~\ref{prop:farey_prop1}, there exists some $m \in \mathbb{N}_0$ such that
	
	\begin{enumerate}[label=-]
		\item $T_x \cup T_0 \subseteq \mathsf{Far}_m$;
		\item for each $q \in T_x \cup T_0$ we have that $q^-$, $q^+ \in \mathsf{Far}_m \backslash (T_x \cup T_0)$;
		\item $0^+$, $1^-\in \mathsf{Far}_m \backslash (T_x \cup T_0)$;
		\item for all $A,B \in \lbrace [0,0^+], [1^-,1] \rbrace \cup \lbrace [q^-,q^+]: q \in T_x \cup T_0\rbrace$, if $A \neq B$ then $A \cap B = \emptyset$.
	\end{enumerate}
	Let $m \in \mathbb{N}_0$ be such that $\mathsf{Far}_m$ satisfies the aforementioned properties. We next define 
	\begin{equation*}
		\begin{aligned}
			T'_x & \coloneq 
			\begin{cases}
				T_x \cup \lbrace q^-: q \in T_x \rbrace \cup \lbrace q^+: q \in T_x \rbrace \cup \lbrace 1^-,1 \rbrace & \text{if $J_1 = \emptyset$;} \\
				T_x \cup \lbrace q^-: q \in T_x \rbrace \cup \lbrace q^+: q \in T_x \rbrace \cup \lbrace 0,0^+,1^-,1 \rbrace & \text{if $J_1 \neq \emptyset$;}
			\end{cases} \\
			T_0' & \coloneq \mathsf{Far}_m \backslash T_x'.
		\end{aligned}
	\end{equation*}
	Taking into consideration Lemma~\ref{lemma:farey_prop2}, for each $q \in \mathsf{Far}_m \backslash \lbrace 1 \rbrace$ we consider the function $f_{[q,q^+]} \colon [q,q^+] \to [0,1]$ defined, for each $a \in [q,q^+]$, by
	\begin{equation*}
		f_{[q,q^+]}(a) \coloneq 
		\begin{cases}
			0 & \text{if $q, q^+ \in T_0'$}; \\
			a & \text{if $q, q^+ \in T_x'$}; \\
			\mathsf{den}(q)\mathsf{num}(q^+)a - \mathsf{num}(q)\mathsf{num}(q^+) & \text{if $q \in T_0'$, $q^+ \in T_x'$}; \\
			-\mathsf{num}(q)\mathsf{den}(q^+)a + \mathsf{num}(q)\mathsf{num}(q^+) & \text{if $q \in T_x'$, $q^+ \in T_0'$}.
		\end{cases}
	\end{equation*} 
	Let $f$ be the function obtained by concatenating all functions $f_{[q,q^+]}$. Then, $f$ is a McNaughton function and $f(1) = 1$. By Theorem~\ref{thm:mcnaughton}, we can calculate a term $\overline{t}$ in the language $\lbrace \cdot,\imp,\top \rbrace$ such that $f = \overline{t}^\mathbf{S}$. By construction,
	\begin{enumerate}[label=-]
		\item for each $q \in T_x$, $f(a) = a$ for all $a \in [q^-,q^+]$;
		\item for each $q \in T_0$, $f(a) = 0$ for all $a \in [q^-,q^+]$;
		\item $f(a) = a$ for all $a \in [1^-,1]$;
		\item if $J_1 = \emptyset$, $f(a) = 0$ for all $a \in [0,0^+]$;
		\item if $J_1 \neq \emptyset$, $f(a) = a$ for all $a \in [0,0^+]$.
	\end{enumerate}
	If we define $t_{\Pi}(x) \coloneq t(x)$, then, as a consequence of Theorem~\ref{thm:local_identities},
	\begin{enumerate}[label=-]
		\item $\mathbf{S}_i \vDash t_\Pi(x) \approx x$ and $\mathbf{S}_j^\omega \vDash t_\Pi(x) \approx x$ for all $i \in I_1$, $j \in J_1$;
		\item $\mathbf{S}_i \vDash t_\Pi(x) \approx x^i$ for all $i \in I_2$;
		\item $\mathbf{S}_j^\omega \vDash t_\Pi(x) \approx x \cdot ((x^{j+1} \imp x^{2(j+1)}) \imp x^{j+1})$ for all $j \in J_2$.
	\end{enumerate}
	To end this proof, we show that $\mathbf{S}^\omega \vDash t_\Pi(x) \approx x$. Observe that, for all $k \in \mathbb{N}$, $\mathbf{S}^\omega$ is isomorphic to the subalgebra of $\mathbf{S}_k^\omega$ given by the set $\lbrace (k,j): j \leq 0 \rbrace$. Since $f(a) = a$ for all $a \in [1^-,1]$, by Theorem~\ref{thm:local_identities} we have that $t_\Pi^{\mathbf{S}_k^\omega}(k,j) = (k,j)$ for all $k \in \mathbb{N}$ and all $j \leq 0$. This concludes the proof.  
\end{proof}

\begin{remark}\label{rmk:multiplicative_separation_conucleus}
	In the previous proof, some of the conditions of $\mathsf{Far}_m$ are not necessary. For example, we do not need to build neighborhoods for the rationals in $\bigcup_{l \in I} S_l \backslash \lbrace 0,1 \rbrace$. On the other hand, if we do need a neighborhood, it is not indispensable that the bounds of these are not elements of $T_x \cup T_0$. Weakening the conditions required for the partition can allow us to choose a smaller $m$ and, therefore, obtain a less complex term. 
	
	Let $f$ be the McNaughton function obtained in the last proof. By  Theorem~\ref{thm:mcnaughton_constructive} we know that
	\begin{equation*}
		f = \bigoplus_{a \in \mathsf{Far}_m} k_a. h_{m,a}
	\end{equation*}	
	where, for each $a \coloneq \frac{c}{d} \in \mathsf{Far}_m$, the integer $k_a \in \mathbb{N}_0$ is such that $f(a) = \frac{k_a}{d}$. Observe that, by construction, $k_a \in \lbrace 0,c \rbrace$ for all $a = \frac{c}{d} \in \mathsf{Far}_m$. Moreover, $k_1 = 1$ and 
	\begin{equation*}
		f = \left(\prod_{\substack{a \in \mathsf{Far}_m \backslash \lbrace 1 \rbrace \\ k_a \neq 0}} (\neg h_{m,a})^{\mathsf{num(a)}} \right) \imp h_{m,1}.
	\end{equation*}
	As a consequence of the definition of each Schauder hat, for each $u \in \lbrace \neg h_{m,a}: a \in \mathsf{Far}_m \backslash \lbrace 1 \rbrace, k_a \neq 0 \rbrace \cup \lbrace h_{m,1} \rbrace$ we have that $u(1) = 1$. By Theorem~\ref{thm:mcnaughton}, we can calculate a term $\overline{u}(x)$ in the language $\lbrace \cdot,\imp,\top \rbrace$ such that $u = \overline{u}^\mathbf{S}$. If we define
	\begin{equation*}
		s(x) \coloneq \left(\prod_{\substack{a \in \mathsf{Far}_m \backslash \lbrace 1 \rbrace \\ k_a \neq 0}} (\overline{\neg h_{m,a}}(x))^{\mathsf{num(a)}} \right) \imp \overline{h_{m,1}}(x).
	\end{equation*}
	then $s(x)$ is a term in the language $\lbrace \cdot,\imp,\top \rbrace$ and $f = s^\mathbf{S}$. These remarks provide a useful method to show some examples.
\end{remark}

We have previously shown that the variety $\mathcal{WH}$ only has trivial conuclei (Proposition~\ref{prop:WH_conuclei}). This leads us to the following result.

\begin{corollary}
	For each subvariety $\mathcal{V}$ of $\mathcal{WH}$ there are, up to equivalence, a finite number of terms that define multiplicative conuclei on $\mathcal{V}$.	
\end{corollary}

We end this section with an example. For the sake of brevity, we are going to avoid some details involving the calculation of the terms (see Remark~\ref{rmk:multiplicative_separation_conucleus}). Consider the reduced triple $(\lbrace 2 \rbrace, \lbrace 1 \rbrace, \emptyset)$, whose corresponding variety is $\mathsf{V}(\mathbf{S}_2,\mathbf{S}_1^\omega)$. A simple calculation yields the following multiplicative separations:
\begin{equation*}
	\begin{aligned}
		\Pi_1 & \coloneq (\emptyset,\lbrace 2 \rbrace, \emptyset,\lbrace 1 \rbrace); \\
		\Pi_2 & \coloneq (\emptyset,\lbrace 2 \rbrace, \lbrace 1 \rbrace,\emptyset); \\
		\Pi_3 & \coloneq (\lbrace 2 \rbrace,\emptyset, \emptyset,\lbrace 1 \rbrace); \\
		\Pi_4 & \coloneq (\lbrace 2 \rbrace,\emptyset, \lbrace 1 \rbrace,\emptyset).
	\end{aligned}
\end{equation*}
To simplify the notation, the multiplicative conucleus corresponding to each multiplicative separation $\Pi_i$ will be denoted by $t_i$. We want terms $t_1(x)$, $t_2(x)$, $t_3(x)$ and $t_4(x)$, in the language $\lbrace \cdot,\imp,\top \rbrace$, such that
\begin{enumerate}[label=-]
	\item $\mathbf{S}_2 \vDash t_1(x) \approx x^2$, $\mathbf{S}_1^\omega \vDash t_1(x) \approx x \cdot ((x^{2} \imp x^4) \imp x^2)$, 
	\item $\mathbf{S}_2 \vDash t_2(x) \approx x^2$, $\mathbf{S}_1^\omega \vDash t_2(x) \approx x$,
	\item $\mathbf{S}_2 \vDash t_3(x) \approx x$, $\mathbf{S}_1^\omega \vDash t_3(x) \approx x \cdot ((x^{2} \imp x^4) \imp x^2)$,
	\item $\mathbf{S}_2 \vDash t_4(x) \approx x$, $\mathbf{S}_1^\omega \vDash t_4(x) \approx x$.
\end{enumerate}
Clearly, we can define $t_4(x) \coloneq x$. On the other hand, if $v(x) \coloneq (x \imp x^2) \imp x$ and $w(x) \coloneq ((x \imp x^2) \imp x) \imp x$, we consider
\begin{equation*}
	\begin{aligned}
		\overline{g}_{\frac{1}{3}}(x) & \coloneq (w(x) \imp v(x)) \imp v(x), \\
		\overline{g}_{\frac{1}{2}}(x) & \coloneq ((v(x) \imp w(x)) \imp ((v(x) \imp w(x)) \cdot x^2)) \imp x^2, \\
		\overline{g}_{\frac{2}{3}}(x) & \coloneq ((w(x) \imp (w(x) \cdot x^2)) \imp x^2) \imp w(x), \\ 
		\overline{g}_{1}(x) & \coloneq x^2 \cdot w(x).
	\end{aligned}
\end{equation*}
Then, $\neg h_{2,a} = \overline{g}_a^\mathbf{S}$ for each $a \in \lbrace \frac{1}{3}, \frac{1}{2}, \frac{2}{3} \rbrace$ and $h_{2,1} = \overline{g}_1^\mathbf{S}$. If we define  
\begin{equation*}
	\begin{aligned}
		t_1(x) & \coloneq \overline{g}_{\frac{2}{3}}(x)^2 \imp \overline{g}_{1}(x), \\
		t_2(x) & \coloneq [\overline{g}_{\frac{1}{3}}(x) \cdot \overline{g}_{\frac{2}{3}}(x)^2] \imp \overline{g}_{1}(x); \\
		t_3(x) & \coloneq [\overline{g}_{\frac{1}{2}}(x) \cdot \overline{g}_{\frac{2}{3}}(x)^2] \imp \overline{g}_{1}(x);
	\end{aligned}
\end{equation*}
then $t_1(x)$, $t_2(x)$ and $t_3(x)$ satisfy the desired properties. The graphs of the corresponding McNaughton function of each nontrivial conucleus are shown in the next figure.

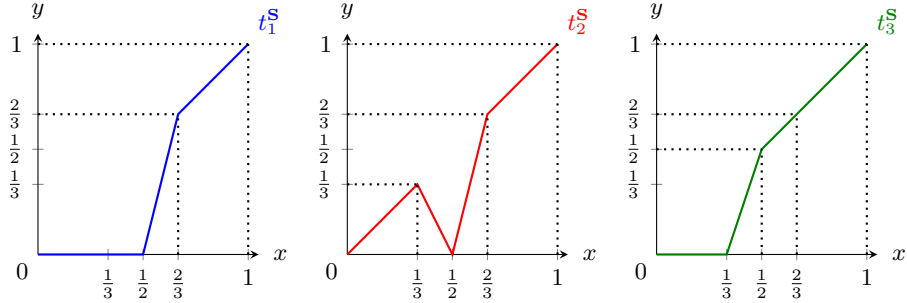
\begin{figure}[htbp]
	\begin{tikzpicture}
		\begin{axis}
			[
			width = 4.5cm,
			height = 4.5cm,
			domain=0:1,
			xmin=0, xmax=1.05,
			ymin=0, ymax=1.05,
			enlarge x limits={abs=0.02},
			enlarge y limits={abs=0.02},
			axis lines=middle,
			axis equal image,
			clip=false,
			xlabel={$x$},
			ylabel={$y$},
			xlabel style={at={(axis description cs:1.025,0)},anchor=west},
			ylabel style={at={(axis description cs:0,1.025)},anchor=south},
			xtick={1/3,1/2,2/3,1},
			ytick={1/3,1/2,2/3,1},
			xticklabels={$\frac{1}{3}$,$\frac{1}{2}$,$\frac{2}{3}$,$1$},
			yticklabels={$\frac{1}{3}$,$\frac{1}{2}$,$\frac{2}{3}$,$1$},
			font=\small
			]
			
			\node[below left] at (axis cs:0,0) {$0$};
			
			\addplot[thick,blue] coordinates {(0,0)	(1/2,0)	(2/3,2/3) (1,1)} node[pos=1, above right] {$t_1^{\mathbf{S}}$};	
			
			\draw[dotted,thick]	(axis cs:0,2/3) -- (axis cs:2/3,2/3);
			\draw[dotted,thick]	(axis cs:2/3,0) -- (axis cs:2/3,2/3);
			\draw[dotted,thick]	(axis cs:0,1) -- (axis cs:1,1);
			\draw[dotted,thick]	(axis cs:1,0) -- (axis cs:1,1);
		\end{axis}
	\end{tikzpicture}
	\hfill
	\begin{tikzpicture}
		\begin{axis}
			[
			width = 4.5cm,
			height = 4.5cm,
			domain=0:1,
			xmin=0, xmax=1.05,
			ymin=0, ymax=1.05,
			enlarge x limits={abs=0.02},
			enlarge y limits={abs=0.02},
			axis lines=middle,
			axis equal image,
			clip=false,
			xlabel={$x$},
			ylabel={$y$},
			xlabel style={at={(axis description cs:1.025,0)},anchor=west},
			ylabel style={at={(axis description cs:0,1.025)},anchor=south},
			xtick={1/3,1/2,2/3,1},
			ytick={1/3,1/2,2/3,1},
			xticklabels={$\frac{1}{3}$,$\frac{1}{2}$,$\frac{2}{3}$,$1$},
			yticklabels={$\frac{1}{3}$,$\frac{1}{2}$,$\frac{2}{3}$,$1$},
			font=\small
			]
			
			\node[below left] at (axis cs:0,0) {$0$};
			
			\addplot[thick,red] coordinates {(0,0) (1/3,1/3) (1/2,0) (2/3,2/3) (1,1)} node[pos=1, above right] {$t_2^{\mathbf{S}}$};	
			
			\draw[dotted,thick]	(axis cs:0,1/3) -- (axis cs:1/3,1/3);
			\draw[dotted,thick]	(axis cs:1/3,0) -- (axis cs:1/3,1/3);
			\draw[dotted,thick]	(axis cs:0,2/3) -- (axis cs:2/3,2/3);
			\draw[dotted,thick]	(axis cs:2/3,0) -- (axis cs:2/3,2/3);
			\draw[dotted,thick]	(axis cs:0,1) -- (axis cs:1,1);
			\draw[dotted,thick]	(axis cs:1,0) -- (axis cs:1,1);
		\end{axis}
	\end{tikzpicture}
	\hfill
	\begin{tikzpicture}
		\begin{axis}
			[
			width = 4.5cm,
			height = 4.5cm,
			domain=0:1,
			xmin=0, xmax=1.05,
			ymin=0, ymax=1.05,
			enlarge x limits={abs=0.02},
			enlarge y limits={abs=0.02},
			axis lines=middle,
			axis equal image,
			clip=false,
			xlabel={$x$},
			ylabel={$y$},
			xlabel style={at={(axis description cs:1.025,0)},anchor=west},
			ylabel style={at={(axis description cs:0,1.025)},anchor=south},
			xtick={1/3,1/2,2/3,1},
			ytick={1/3,1/2,2/3,1},
			xticklabels={$\frac{1}{3}$,$\frac{1}{2}$,$\frac{2}{3}$,$1$},
			yticklabels={$\frac{1}{3}$,$\frac{1}{2}$,$\frac{2}{3}$,$1$},
			font=\small
			]
			
			\node[below left] at (axis cs:0,0) {$0$};
			
			\addplot[thick,green!50!black] coordinates {(0,0) (1/3,0) (1/2,1/2) (2/3,2/3) (1,1)} node[pos=1, above right] {$t_3^{\mathbf{S}}$};	
			
			\draw[dotted,thick]	(axis cs:0,1/2) -- (axis cs:1/2,1/2);
			\draw[dotted,thick]	(axis cs:1/2,0) -- (axis cs:1/2,1/2);
			\draw[dotted,thick]	(axis cs:0,2/3) -- (axis cs:2/3,2/3);
			\draw[dotted,thick]	(axis cs:2/3,0) -- (axis cs:2/3,2/3);
			\draw[dotted,thick]	(axis cs:0,1) -- (axis cs:1,1);
			\draw[dotted,thick]	(axis cs:1,0) -- (axis cs:1,1);
		\end{axis}
	\end{tikzpicture}
\end{figure}

\section{Mutltiplicative conuclei on varieties of basic hoops and BL-algebras}\label{sec:conuclei_on_BH_and_BL}

In this section we are going to focus on terms that define (multiplicative) conuclei on varieties of basic hoops and BL-algebras. Among some results, we prove that this problem is equivalent to finding terms that define (multiplicative) conuclei on some variety or some pair of varieties of Wajsberg hoops (Theorem~\ref{thm:basic_varieties_conuclei} and Theorem~\ref{thm:BL_conuclei}). 

The next result states a crucial property of (multiplicative) conuclei defined on nontrivial totally ordered hoops. It will be our starting point and an important tool to characterize arbitrary (multiplicative) conuclei on varieties of basic hoops and BL-algebras.

\begin{proposition}\label{prop:conucleus_on_hoop_chain}
	Let $\mathbf{A} = \bigoplus_{i \in I} \mathbf{A}_i$ be a nontrivial totally ordered hoop and $\delta\colon A \to A$ such that $\delta(A_i) \subseteq A_i$ for all $i \in I$. The following conditions are equivalent:
	\begin{enumerate}[label=\rm{(\roman*)}]
		\item $\delta$ is a (multiplicative) conucleus on $\mathbf{A}$;
		\item $\ensuremath{\left. \delta \right|_{A_i}}$ is a (multiplicative) conucleus on $\mathbf{A}_i$ for all $i \in I$.
	\end{enumerate}
\end{proposition}
\begin{proof}
	We are going to restrict ourselves to the proof for conuclei. The proof for multiplicative conuclei is almost identical. First observe that $\delta(\top) \in \bigcap_{i \in I} A_i = \lbrace \top \rbrace$, that is, $\delta(\top) = \top$.
	
	We first assume that $\delta$ is a conucleus on $\mathbf{A}$. If $a,b \in A_i$, then the conditions $\delta(a) \leq a$, $\delta(a \land b) \leq \delta(a)$ and $\delta(\delta(a)) = \delta(a)$ are immediate. Moreover, taking into consideration the definition of ordinal sums and the fact that $\delta(a),\delta(b) \in A_i$,
	\begin{equation*}
		\delta(a) \cdot_i \delta(b) = \delta(a) \cdot \delta(b) \leq \delta(a \cdot b) = \delta(a \cdot_i b).
	\end{equation*}              
	
	Now we assume that $\ensuremath{\left. \delta \right|_{A_i}}$ is a conucleus on $\mathbf{A}_i$ for all $i \in I$. Let $a,b \in A$ and $i,j \in I$ such that $a \in A_i$, $b \in A_j$. The conditions $\delta(a) \leq a$ and $\delta(\delta(a)) = \delta(a)$ are immediate. Particularly, observe that $\delta(c) = \top$ if and only if $c = \top$.  To prove the monotonicity, let us assume that $a \leq b$. If $b = \top$, then there is nothing to prove. If $b < \top$, then $i \leq j$, $\delta(a) \in A_i \backslash \lbrace \top \rbrace$ and $\delta(b) \in A_j \backslash \lbrace \top \rbrace$. If $i = j$ then, since $\ensuremath{\left. \delta \right|_{A_i}}$ is a conucleus, $\delta(a) \leq \delta(b)$. If $i < j$, then the definition of ordinal sum implies $\delta(a) \leq \delta(b)$. To prove the inequality $\delta(a) \cdot \delta(b) \leq \delta(a \cdot b)$ we analyze the same cases as before. If $b = \top$, then there is nothing to prove. If $b < \top$ and $i = j$, then
	\begin{equation*}
		\delta(a) \cdot \delta(b) = \delta(a) \cdot_i \delta(b) \leq \delta(a \cdot_i b) = \delta(a \cdot b).
	\end{equation*}
	On the other hand, if $i < j$, $\delta(a) \cdot \delta(b) = \delta(a) = \delta(a \cdot b)$.  
\end{proof}

Taking into consideration the previous result, we can proceed with one of the main result of this section.

\begin{theorem}\label{thm:basic_varieties_conuclei}
	Let $\mathcal{V}  \coloneq \mathsf{V}(\mathcal{K})$, where $\mathcal{K}$ is some nonempty class of nontrivial totally ordered hoops. If $t(x)$ is a term in the language $\lbrace \cdot,\imp,\top \rbrace$, the following are equivalent:
	\begin{enumerate}[label=\rm{(\roman*)}]
		\item $t(x)$ is a \textup{(}multiplicative\textup{)} conucleus on $\mathcal{V}$;
		\item $t(x)$ is a \textup{(}multiplicative\textup{)} conucleus on
		\begin{equation*}
			\textstyle \mathcal{U} \coloneq \mathsf{V}(\lbrace \mathbf{A}_j \in \mathcal{WH}: \bigoplus_{i \in I} \mathbf{A}_i \in \mathcal{K}, j \in I \rbrace).
		\end{equation*}
	\end{enumerate}
\end{theorem}
\begin{proof}
	The implication (i) $\Rightarrow$ (ii) is immediate, because $\mathbf{A}_j$ is a subalgebra of $\bigoplus_{i \in I} \mathbf{A}_i$ for all $j \in I$.
	
	Now suppose that $t(x)$ is a (multiplicative) conucleus on $\mathcal{U}$ and let $\mathbf{A} \coloneq \bigoplus_{i \in I} \mathbf{A}_i \in \mathcal{K}$. If $j \in I$, then $\mathbf{A}_j$ is a subalgebra of $\mathbf{A}$ and $\ensuremath{\left. t^\mathbf{A} \right|_{A_j}} = t^{\mathbf{A}_j}$. Therefore $t^{\mathbf{A}}(A_j) \subseteq A_j$ and, as a consequence of the hypothesis, $\ensuremath{\left. t^\mathbf{A} \right|_{A_j}}$ is a (multiplicative) conucleus on $\mathbf{A}_j$. By Proposition~\ref{prop:conucleus_on_hoop_chain}, $t^\mathbf{A}$ is a (multiplicative) conucleus on $\mathbf{A}$. This concludes the proof.  
\end{proof}

The previous theorem states that the problem of finding (multiplicative) conuclei on varieties of basic hoops reduces to calculating all (multiplicative) conuclei on some variety of Wajsberg hoops. Our objective is to obtain a similar result for varieties of BL-algebras. To reach this goal, we need the following result.

\begin{proposition}\label{prop:bl_chain_term_conucleus}
	Let $\mathcal{V}$ be a subvariety of $\mathcal{BL}$ and $t(x)$ a \textup{(}multiplicative\textup{)} conucleus on $\mathcal{V}$ such that $t(\top) \approx \top$. If $\mathbf{A} \coloneq (\bigoplus_{i \in I} \mathbf{A}_i)^+ \in \mathcal{V}_{\textnormal{fsi}}$, 
	\begin{enumerate}[label=\rm{(\arabic*)}]
		\item $t^\mathbf{A}(A_i) \subseteq A_i$ for all $i \in I$;
		\item $\ensuremath{\left. t^\mathbf{A} \right|_{A_i}}$ is a \textup{(}multiplicative\textup{)} conucleus on $\mathbf{A}_i$ for all $i \in I$.
	\end{enumerate}
\end{proposition}
\begin{proof}
	Item (1) is an immediate consequence of Corollary~\ref{cor:BL_terms}. Item (2) follows from Proposition~\ref{prop:conucleus_on_hoop_chain}.  
\end{proof}

We continue with the second main result of this section, a characterization of (multiplicative) conuclei on varieties of BL-algebras.

\begin{theorem}\label{thm:BL_conuclei}
	Let $\mathcal{V}  \coloneq \mathsf{V}(\mathcal{K})$, where $\mathcal{K}$ is some nonempty class of nontrivial BL-chains. If $t(x)$ is a term in the language $\lbrace \cdot,\imp,\bot,\top \rbrace$, the following are equivalent:
	\begin{enumerate}[label=\rm{(\roman*)}]
		\item $t(x)$ is a \textup{(}multiplicative\textup{)} conucleus on $\mathcal{V}$;
		\item $\mathcal{V} \vDash t(x) \approx \bot$ or $\mathcal{V} \vDash t(x) \approx t_0(\neg\neg x) \cdot t_1(\neg\neg x \imp x)$ for some terms $t_0(x)$ and $t_1(x)$ in the language $\lbrace \cdot,\imp,\top \rbrace$ such that $t_0(x)$ is a \textup{(}multiplicative\textup{)} conucleus on 
		\begin{equation*}
			\mathcal{V}_0 \coloneq \mathsf{V}(\lbrace \mathbf{A}_0: \textstyle (\bigoplus_{i \in I} \mathbf{A}_i)^+ \in \mathcal{K}, 0 \in I \rbrace)
		\end{equation*}
		and $t_1(x)$ is a \textup{(}multiplicative\textup{)} conucleus on
		\begin{equation*}
			\mathcal{V}_1 \coloneq \mathsf{V}(\lbrace \mathbf{A}_j: \textstyle (\bigoplus_{i \in I} \mathbf{A}_i)^+ \in \mathcal{K}, j \in I \backslash \lbrace 0 \rbrace \rbrace \cup \lbrace \mathbf{S}_0 \rbrace).	
		\end{equation*}
	\end{enumerate}
\end{theorem}
\begin{proof}
	First assume that $t(x)$ is a (multiplicative) conucleus on $\mathcal{V}$. If $t(\top) \approx \bot$, then $\mathcal{V} \vDash t(x) \approx \bot$. If $t(\top) \approx \top$ then, by Corollary~\ref{cor:BL_terms} and Proposition~\ref{prop:BLterms_equivalence}, there exist terms $t_0(x)$ and $t_1(x)$ in the language $\lbrace \cdot,\imp,\top \rbrace$ such that, for all $\mathbf{A} \coloneq (\bigoplus_{i \in I} \mathbf{A}_i)^+ \in \mathcal{K}$,
	\begin{enumerate}[label=-]
		\item $\ensuremath{\left. t^\mathbf{A} \right|_{A_0}} = \ensuremath{\left. t_0^\mathbf{A} \right|_{A_0}} = t_0^\mathbf{A_0}$
		\item $\ensuremath{\left. t^\mathbf{A} \right|_{A_i}} = \ensuremath{\left. t_1^\mathbf{A} \right|_{A_i}} = t_1^\mathbf{A_i}$ for all $i \in I \backslash \lbrace 0 \rbrace$;
		\item $\mathbf{A} \vDash t(x) \approx t_0(\neg\neg x) \cdot t_1(\neg\neg x \imp x)$.
	\end{enumerate}
	Furthermore, by Corollary~\ref{prop:bl_chain_term_conucleus}, $\ensuremath{\left. t^\mathbf{A} \right|_{A_i}}$ is a (multiplicative) conucleus on $\mathbf{A}_i$ for all $i \in I$. Since $\mathbf{A} \in \mathcal{K}$ is arbitrary, $t_0(x)$ is a (multiplicative) conucleus on $\mathcal{V}_0$ and $t_1(x)$ is a (multiplicative) conucleus on $\mathcal{V}_1$. 
	
	Now assume that (ii) holds. If $\mathcal{V} \vDash t(x) \approx \bot$ there is nothing to prove. Let us assume that $\mathcal{V} \vDash t(x) \approx t_0(\neg\neg x) \cdot t_1(\neg\neg x \imp x)$ with $t_0(x)$, $t_1(x)$ (multiplicative) conuclei on $\mathcal{V}_0$ and $\mathcal{V}_1$, respectively. By Proposition~\ref{prop:BLterms_equivalence}, for all $\mathbf{A} \coloneq (\bigoplus_{i \in I} \mathbf{A}_i)^+ \in \mathcal{K}$, $\ensuremath{\left. t^\mathbf{A} \right|_{A_0}} = \ensuremath{\left. t_0^\mathbf{A} \right|_{A_0}} = t_0^\mathbf{A_0}$ and $\ensuremath{\left. t^\mathbf{A} \right|_{A_i}} = \ensuremath{\left. t_1^\mathbf{A} \right|_{A_i}} = t_1^\mathbf{A_i}$ for all $i \in I \backslash \lbrace 0 \rbrace$. As a consequence of the hypothesis we conclude that $\ensuremath{\left. t^\mathbf{A} \right|_{A_i}}$ is a (multiplicative) conucleus on $\mathbf{A}_i$ for all $i \in I$. By Proposition~\ref{prop:conucleus_on_hoop_chain}, $t^\mathbf{A}$ is a (multiplicative) conucleus on $\bigoplus_{i \in I} \mathbf{A}_i$ (and therefore on $\mathbf{A}$). Since $\mathbf{A} \in \mathcal{K}$ is arbitrary, $t(x)$ is a (multiplicative) conucleus on $\mathcal{V}$.  
\end{proof}

We end this section with some examples. The first examples we present are in the context of varieties of basic hoops.
\begin{example}
	To begin, we define $\mathcal{V} \coloneq \mathsf{V}(\mathbf{S}_n \oplus \mathbf{S}_n)$ for some $n \geq 2$. In this case, it is clear that $\mathcal{U} = \mathsf{V}(\mathbf{S}_n)$. By Proposition~\ref{prop:Sn_multiplicative_conuclei} and Theorem~\ref{thm:basic_varieties_conuclei}, $t(x)$ is a multiplicative conucleus on $\mathcal{V}$ if and only if $\mathcal{V} \vDash t(x) \approx s(x)$ for some $s(x) \in \lbrace x^n, x \rbrace$. 
	
	Analogously, by Proposition~\ref{prop:cancellative_conuclei}, $t(x)$ is a conucleus on $\mathsf{V}(\mathbf{S}^\omega \oplus \mathbf{S}^\omega)$ if and only if it is trivial. Moreover, by Proposition~\ref{prop:Snomega_multiplicative_conuclei}, $t(x)$ is a multiplicative conucleus on $\mathsf{V}(\mathbf{S}_n^\omega \oplus \mathbf{S}_n^\omega)$ ($n \in \mathbb{N}$) if and only if $\mathsf{V}(\mathbf{S}_n^\omega \oplus \mathbf{S}_n^\omega) \vDash t(x) \approx s(x)$ for some $s(x) \in \lbrace x \cdot ((x^{n+1} \imp x^{2(n+1)}) \imp x^{n+1}), x \rbrace$.
	
	For the last example, let $\mathcal{K}$ be an arbitrary class of nontrivial totally ordered hoops and let $\mathcal{V} \coloneq \mathsf{V}(\mathcal{K} \cup \lbrace \mathbf{S} \rbrace)$. In this case, observe that $\mathcal{U} = \mathcal{WH}$. By Proposition~\ref{prop:WH_conuclei}, $t(x)$ is a conucleus on $\mathcal{V}$ if and only if it is trivial.
\end{example}

The last examples given in this section are in the context of varieties of BL-algebras.
\begin{example}
	Let $n,m \in \mathbb{N}$ and $\mathcal{K} \coloneq \lbrace (\mathbf{S}_n^\omega \oplus \mathbf{S}_m)^+ \rbrace$. In this case, $\mathcal{V}_0 = \mathsf{V}(\mathbf{S}_n^\omega)$ and $\mathcal{V}_1 = \mathsf{V}(\mathbf{S}_m)$. By Proposition~\ref{prop:Sn_multiplicative_conuclei}, Proposition~\ref{prop:Snomega_multiplicative_conuclei} and Theorem~\ref{thm:BL_conuclei}, $t(x)$ is a multiplicative conucleus on $\mathsf{V}(\mathcal{K})$ if and only if $\mathsf{V}(\mathcal{K}) \vDash t(x) \approx s(x)$, where $s(x)$ is some term in the following list:
	\begin{enumerate}[label=-]
		\item $\bot$;
		\item $\neg\neg x \cdot (\neg\neg x \imp x)$ ($\approx x$);
		\item $\neg\neg x \cdot (((\neg\neg x)^{n+1} \imp (\neg\neg x)^{2(n+1)}) \imp (\neg\neg x)^{n+1}) \cdot (\neg\neg x \imp x)$;
		\item $\neg\neg x \cdot (\neg\neg x \imp x)^m$;
		\item $\neg\neg x \cdot (((\neg\neg x)^{n+1} \imp (\neg\neg x)^{2(n+1)}) \imp (\neg\neg x)^{n+1}) \cdot (\neg\neg x \imp x)^m$.
	\end{enumerate}
	
	To end, let $n,m \in \mathbb{N}$ and $\mathcal{K} \coloneq \lbrace (\mathbf{S}_n \oplus \mathbf{S}_m^\omega)^+ \rbrace$. In this case, $t(x)$ is a multiplicative conucleus on $\mathsf{V}(\mathcal{K})$ if and only if $\mathsf{V}(\mathcal{K}) \vDash t(x) \approx s(x)$, where $s(x)$ is some term in the following list:
	\begin{enumerate}[label=-]
		\item $\bot$;
		\item $\neg\neg x \cdot (\neg\neg x \imp x)$ ($\approx x$);
		\item $(\neg\neg x)^n \cdot (\neg\neg x \imp x)$;
		\item $\neg\neg x \cdot (\neg\neg x \imp x) \cdot (((\neg\neg x \imp x)^{m+1} \imp (\neg\neg x \imp x)^{2(m+1)}) \imp (\neg\neg x \imp x)^{m+1})$;
		\item $(\neg\neg x)^n \cdot (\neg\neg x \imp x) \cdot (((\neg\neg x \imp x)^{m+1} \imp (\neg\neg x \imp x)^{2(m+1)}) \imp (\neg\neg x \imp x)^{m+1})$.
	\end{enumerate}
\end{example}

	\newpage
	
	Sebasti\'an Buss
	
	\noindent Instituto de Matem\'atica (INMABB), Departamento de Matem\'atica, Universidad Nacional del Sur (UNS)-CONICET, Bah\'ia Blanca, Argentina.
	
	\noindent sbuss94@gmail.com
	
	\
	
	Diego Casta\~no
	
	\noindent Departamento de Matem\'atica, Universidad Nacional del Sur (UNS), Bah\'ia Blanca, Argentina. \\
	Instituto de Matem\'atica (INMABB), Departamento de Matem\'atica, Universidad Nacional del Sur (UNS)-CONICET, Bah\'ia Blanca, Argentina.
	
	\noindent diego.castano@uns.edu.ar
	
	\
	
	Jos\'e Patricio D\'iaz Varela
	
	\noindent Departamento de Matem\'atica, Universidad Nacional del Sur (UNS), Bah\'ia Blanca, Argentina. \\
	Instituto de Matem\'atica (INMABB), Departamento de Matem\'atica, Universidad Nacional del Sur (UNS)-CONICET, Bah\'ia Blanca, Argentina.
	
	\noindent usdiavar@criba.edu.ar
	
\end{document}